\documentclass[listof=flat]{scrartcl}

\usepackage{graphicx}
\usepackage{color, colortbl}
\usepackage{scrextend}
\usepackage{multirow}
\usepackage{caption}
\usepackage{tabularx}
\usepackage{aligned-overset}
\usepackage{amssymb,amsmath,amsthm,mathtools, bbm, mathrsfs} 
\usepackage{nicefrac}
\usepackage[margin=1.25in]{geometry}
\usepackage[utf8]{inputenc}
\usepackage{fancyhdr}
\usepackage{enumerate}
\usepackage[ngerman,english]{babel}
\usepackage{csquotes}
\usepackage[T1]{fontenc}
\usepackage{bbm}
\usepackage{tabularx}
\usepackage{mathdots}
\usepackage[makeroom]{cancel}
\usepackage{tikz}
\usepackage[draft]{todonotes}
\usepackage{tikz-imagelabels}
\usepackage{caption}
\usepackage[labelformat=simple]{subcaption}

\usepackage[wby]{callouts}

\usepackage{extpfeil}
\usepackage{comment}
\usepackage{appendix}

\usepackage{hyperref}

\newtheorem{theorem}{Theorem}[section]

\newtheorem{rem}[theorem]{Remark}
\newtheorem{lemma}[theorem]{Lemma}
\newtheorem{proposition}[theorem]{Proposition}
\newcommand{\interior}{\operatorname{int}}
\newcommand{\dist}{\operatorname{dist}}

\newcommand{\diam}{\operatorname{diam}}
\newcommand{\R}{\mathbb{R}}
\newcommand{\rot}{\operatorname{rot}}

\makeatletter
\newcommand{\pushright}[1]{\ifmeasuring@#1\else\omit\hfill$\displaystyle#1$\fi\ignorespaces}
\newcommand{\pushleft}[1]{\ifmeasuring@#1\else\omit$\displaystyle#1$\hfill\fi\ignorespaces}
\makeatother

\theoremstyle{definition} 
\newtheorem{mydef}{Definition}[section]

\begin{document}

\title{Chemical distance in the Poisson Boolean model with regularly varying diameters}

\author{Peter Gracar and Marilyn Korfhage}

\maketitle
\begin{abstract}\noindent\textbf{Abstract:}
We study the Poisson Boolean model with convex bodies which are rotation-invariant distributed. We assume that the convex bodies have regularly varying diameters with indices $-\alpha_1\geq \dots\geq-\alpha_d$ where $\alpha_k >0$ for all $k\in\{1,\dots,d\}.$ It is known that a sufficient condition for the robustness of the model, i.e.\ the union of the convex bodies has an unbounded connected component no matter what the intensity of the underlying Poisson process is, is that there exists some $k\in\{1,\dots,d\}$ such that $\alpha_k<\min\{2k,d\}$. 
To avoid that this connected component covers all of $\mathbb{R}^d$ almost surely we also require $\alpha_k> k$ for all $k\in\{1,\dots,d\}$. We show that under these assumptions, the chemical distance of two far apart vertices $\mathbf{x}$ and $\mathbf{y}$ behaves like $c\log\log|x-y|$ as $|x-y|\rightarrow \infty$, with an explicit and very surprising constant $c$ that depends only on the model parameters. We furthermore show that if there exists $k$ such that $\alpha_k\leq k$, the chemical distance is smaller than $c\log\log|x-y|$ for all $c>0$ and that if $\alpha_k\geq\min\{2k,d\}$ for all $k$, it is bigger than $c\log\log|x-y|$ for all $c>0$.
\end{abstract}

\renewcommand{\contentsname}{Contents}
\tableofcontents

\section{Introduction and result}\label{Sect:Introduction}

A key topic in percolation theory and, more broadly, the study of geometrically embedded random graphs is the relationship between the Euclidean distance of two points and their graph distance, commonly referred to as the chemical distance. This problem has been extensively studied since the work of Grimmett and Marstrand~\cite{Grimmett1990}, particularly in the context of Bernoulli percolation, with contributions from Antal and Pisztora~\cite{Antal1996}, as well as Garet and Marchand~\cite{Garet2004, Garet2007}. Research has also extended to models incorporating long-range interactions, such as random interlacements (\v{C}ern\'y and Popov~\cite{Cerny2012}), its vacant set, and the Gaussian free field (Drewitz et al.~\cite{Drewitz2014}).

In the supercritical phase of these models, when points belong to the unbounded connected component and are sufficiently far apart, their Euclidean and chemical distances are typically of the same order. The general conditions under which the percolation models in $\mathbb{Z}^d$ exhibit this behaviour are discussed in~\cite{Drewitz2014}. However, introducing additional long-range edges can significantly alter this relationship, leading to cases where the graph distance scales as a power of the logarithm or even as an iterated logarithm of the Euclidean distance. When the latter occurs, the graph is said to be \emph{ultrasmall}. This phenomenon has been observed, for example, in geometric random graphs with long edges and scale-free degree distribution \cite{Gracar2022a}.

In this paper, we study instead the chemical distance in the Poisson Boolean model with convex grains and establish a universal limit theorem for typical distances in such graphs when the model is robust (see below).
We start with a homogeneous Poisson point process $\mathscr{P}$ on $\mathbb{R}^d$ with $d\geq2$ and positive intensity $u$. We attach to every point $x\in\mathscr{P}$ a mark $\tilde{C}_x$ which is an i.i.d.\ copy of a random, rotation invariant distributed convex body $C$, where $C$ is such that it includes a ball of fixed radius $\epsilon>0$. Setting $C_x:=x+\tilde{C}_x$, one can look at the union of these convex bodies, defined as
\begin{equation*}
    \mathscr{C} = \bigcup\limits_{x\in\mathscr{P}} C_x.
\end{equation*}
It is also possible to naturally define a graph $\mathscr{G}:=(\mathscr{P},\mathscr{E})$ with the edge set $\mathscr{E}\subset\mathscr{P}\times \mathscr{P}$ as follows. We say an edge between two distinct points $x,y\in\mathscr{P}$ exists if and only if the corresponding convex bodies intersect, i.e.\ $C_x\cap C_y\neq\emptyset$. This model has been the focus of study by several authors, see for example the survey of Boolean models in \cite{Meester1996a,Hug2023}. 
Studying the behaviour of $\mathscr{C}$ as a result of properties of $C$, different connectivity regimes arise, namely a \emph{dense}, a \emph{robust} and a \emph{non-robust} regime. We say the model is dense if and only if $\mathscr{C}=\mathbb{R}^d$ almost surely. In \cite{Gracar2024a} it is shown that this is a property that does not depend on the intensity $u$ and it occurs if and only if $\operatorname{Vol}(C)\not\in\mathcal{L}^1$ and that it is a characteristic which gives information about the corresponding random graph $\mathscr{G}$, the covering of the process, and $\mathscr{C}$. Above, $\operatorname{Vol}(A)$ is the standard $d$ dimensional Lebesgue measure of a set $A\subset\mathbb{R}^d$. We can observe robustness, i.e.\ the model being robust, if and only if for every $u>0$ there exists an unbounded connected component of $\mathscr{C}$ almost surely. In all other cases we say the model is non-robust, that is, there exists a critical intensity $u_c>0$ such that for $u<u_c$ there exists no unbounded connected component of $\mathscr{C}$. 
Beyond \cite{Gracar2024a}, this behaviour has also been studied in \cite{Hall1985,Gouere2008} where the authors studied the Boolean model with balls, in \cite{Roy2002,Gouere2023} where the Boolean model with convex shapes, and finally in \cite{Teixeira2017} where ellipses percolation in $2$ dimensions was studied.

In this work, we focus on the graph $\mathscr{G}$ and investigate how the chemical distance between two points of $\mathscr{G}$ compares to their euclidean distance. This relationship has been studied for related models in \cite{Gracar2022a, Hao2023} and for our case with $d=2$ and the convex body an ellipse with a heavy tailed long axis and a fixed length short axis in \cite{Hilario2021}.

The model we will work with was first introduced in \cite{Gracar2024a} under the name \emph{Poisson-Boolean base model with regularly varying diameters} and we now proceed to formally define it. In the following definition and throughout this work, we use $|\cdot|$ to denote the euclidean norm of a vector in $\mathbb{R}^d$. Note that the first definition introduces a sequence of diameters and that this is the only meaning of diameters in this paper.
\begin{mydef}
     Let $K\subset\mathbb R^d$ be a convex body, i.e.\ a compact convex set with nonempty interior. We define a decreasing sequence of diameters \smash{$D_K^{_{(1)}},\dots,D_K^{_{(d)}}$} iteratively, to roughly describe the shape of the convex body, as follows. The \emph{first diameter}, or just \emph{diameter} of $K$, is the classical diameter of a set in the euclidean space, i.e.\
\begin{equation*}
    D_K^{_{(1)}}:= \diam(K) = \max\{|x-y|\,:\, x,y \in K\}.
\end{equation*}
Let $x,y$ be any (measurable) choice from $K$ such that $|x-y|=\diam(K)$. We define $p^{_{(1)}}_{K}:=\tfrac{x-y}{|x-y|}$ as the \emph{orientation} of $D_K^{_{(1)}}$. Defining now $P_H(B)$ as the orthogonal projection of $B\subset\mathbb{R}^d$ onto the linear subspace $H\subset \mathbb{R}^d$ we define the \emph{second diameter} as 
\begin{equation*}  
    D_K^{_{(2)}}:=\diam\Bigl(P_{H_{p^{_{(1)}}_K}}(K)\Bigr),
\end{equation*} 
where \smash{$H_{p^{_{(1)}}_K}$} is a hyperplane (without loss of generality containing the origin) perpendicular to the orientation $p^{_{(1)}}_{K}$. 
We define the orientation of $D_K^{_{(2)}}$ by $p^{_{(2)}}_{K}\!\in H_{p^{_{(1)}}_K}$ analogously and let \smash{$H_{p^{_{(2)}}_K}$} be a hyperplane (without loss of generality containing the origin) in \smash{$ H_{p^{_{(1)}}_K}$} that is perpendicular to $p^{_{(2)}}_{K}$.
For any $2\leq i <d$ and given \smash{$H_{p^{_{(i)}}_K}$}, the $(i+1)$st diameter is 
\begin{equation*}
    D_K^{_{(i+1)}} := \diam \Bigl( P_{H_{p^{_{(i)}}_K}}(K)\Bigr).
\end{equation*}
We also define $p^{_{(i+1)}}_K$ again to be the corresponding orientation.
\end{mydef}

We denote by $\mathcal{C}^d$ the space of convex bodies in $\mathbb{R}^d$ equipped with the Hausdorff metric. Let $\mathbb{P}_C$ be the law on $\mathcal{C}^d\times \mathbb{R}^d$ such that for $\mathbb{P}_C$-almost every $(C,m)$ the point $m$ is in the $\epsilon$-interior of $C$, which exists since we assume that there exists a fixed $\epsilon>0$ for which the ball of radius $\epsilon$ is completely included in $C$. Additionally we assume that $\mathbb{P}_C$ is rotation invariant with respect to rotations around the origin. The Poisson-Boolean base model with regularly varying diameters is then defined as follows. 

\begin{mydef}  \label{Def:model}
    The \emph{Poisson-Boolean base model with regularly varying diameters} is the Poisson point process on $$\mathcal{S}:=\mathbb{R}^d\times \big(\mathcal C^d\times \mathbb{R}^d \big) $$
    with intensity $$  u \, \lambda \otimes\mathbb P_C $$
    where $\lambda$ is the Lebesgue measure. In addition to that we assume that $\mathbb{P}_C$ is given such that the marginal distributions of the diameters $D_C^{(1)},\dots,D_C^{(d)}$ of $C$ are regularly varying with indices $-\alpha_1\geq -\alpha_2\geq\dots\geq-\alpha_d$ with $\alpha_k>0$ for all $k\in\{1,\dots,d\}$, i.e.\ 
    \begin{equation*}
	\lim\limits_{r \rightarrow \infty} \frac{\mathbb{P}(D_C^{_{(i)}}\geq cr)} {\mathbb{P}(D_C^{_{(i)}}\geq r)} = c^{-\alpha_i} \text{ for all $c>1$.}
	\end{equation*}   
    Setting $\alpha_i=\infty$ for some $i\in\{1,\dots,d\}$ also covers the case of bounded diameters $D_C^{_{(i)}}$. 
    
    We denote by $\mathcal{X}$ the corresponding point process and call any $\mathbf{x}\in\mathcal{X}$ a \emph{vertex} and its first component~$x\in\mathbb R^d$
    its \emph{location}. We denote the Poisson point process of these locations $\mathscr P$. The second component of $\mathbf{x}$ is denoted by $(\tilde C_x, m_x)$ and we call $\tilde C_x$ the \emph{grain at $x$ with interior point $m_x$}. We denote its diameters by $D_x^{_{(1)}},\dots,D_x^{_{(d)}}$
    and the corresponding directions by $p_x^{_{(1)}},\dots,p_x^{_{(d)}}$. $C_x$ for $x\in\mathscr{P}$ is defined as previously, i.e.\ $C_x:=x+\tilde C_x$. 
\end{mydef}

For this model it is shown in \cite{Gracar2024a} that it is robust if there exists $k\in\{1,\dots,d\}$ such that $\alpha_k<\min\{2k,d\}$ and non-robust if either $\alpha_k>2k$ for all $k\in\{1,\dots,d\}$ and $\operatorname{Vol}(C)\in\mathcal{L}^2$, or if $D^{_{(1)}}\in\mathcal{L}^d$.
The main result of this paper gives us the behaviour of the chemical distance in the robust (but not dense) regime. For that let 
\begin{equation*}
    M:= \big\{k\in\{1,\dots,d-1\}\,: \, \alpha_k\in(k,\min\{2k, d\})\big\},
\end{equation*}
which can intuitively be thought of as the index set of tail exponents that are sufficiently small for their corresponding diameters to be able to affect the chemical distances in the graph. We define $\mathbb{P}_A$ for $A=\{x_1,...,x_n\}\subset \mathbb{R}^d$ for $n\in\mathbb{N}$ to be the probability measure of the Poisson point process under the condition that $x_1,...,x_n\in\mathscr{P}$, i.e.\ the Palm version of the process and $\dist(\mathbf{x},\mathbf{y})$ to be the chemical distance, i.e.\ length of the shortest path between $\mathbf{x},\mathbf{y}\in\mathcal{X}$. Then, the following holds.

\begin{theorem}\label{Theorem}
In the Poisson-Boolean base model with regularly varying diameters with $u>0$ and $\kappa:= \underset{s\in M}{\operatorname{argmax}} \frac{\min\{d-s,s\}}{\alpha_s -s}$ we have for $x,y\in\mathscr{P}$ and $\delta>0$  in the case $M\neq\emptyset$ and $\alpha_k>k$ for all $k\in\{1,\dots,d\}$ that
\begin{equation*}
    \lim\limits_{|x-y|\rightarrow \infty} \mathbb{P}_{x,y}\Bigl( \tfrac{(2-\delta)\log\log|x-y|} {\log\bigl( \frac{\min\{d-\kappa, \kappa \}}{\alpha_{\kappa}-\kappa}\bigr) }\, \leq \,\dist(\mathbf{x},\mathbf{y}) \, \leq\,  \tfrac{(2+\delta)\log\log|x-y|} {\log\bigl( \frac{\min\{d-\kappa, \kappa \}}{\alpha_{\kappa}-\kappa} \bigr)}\, \, \Big|\, \, x\leftrightarrow y \Bigr) =1.
\end{equation*}
Furthermore if there exists $k\in\{1,\dots,d\}$ such that $\alpha_k\leq k$ we have that $\dist(\mathbf{x},\mathbf{y})$ is smaller than $c\log\log|x-y|$ for all $c>0$ with high probability. In all other cases, namely $\alpha_k\geq\min\{2k,d\}$ for all $k\in\{1,\dots,d\}$, we have that $\dist(\mathbf{x},\mathbf{y})$ is bigger than $c\log\log|x-y|$ for all $c>0$. 
\end{theorem}

In the theorem, $x\leftrightarrow y$ denotes the event that there exists a path in $\mathscr{G}$ between $x$ and $y$, or equivalently that $x$ and $y$ belong to the same connected component of $\mathscr{C}$.

\paragraph{Discussion of result.} Before proceeding with the rest of the paper, we quickly discuss the unusual nature of the scaling of chemical distances. More precisely, we comment on the scaling factor $\log\bigl( \frac{\min\{d-\kappa, \kappa \}}{\alpha_{\kappa}-\kappa}\bigr)/2$ that relates the chemical to the euclidean distance. The most illustrative comparison comes from so-called \emph{scale-free networks}. These networks tend not to be modelled as spatial graphs, so chemical distances must be given in terms of the number of vertices in the graph rather than the euclidean distance, with the number of vertices $N$ roughly corresponding to $|x-y|^d$. Then, the mean-field analogue of the classical Boolean model with Pareto($-\gamma/d$) radius distribution exhibits ultrasmallness with the scaling factor $\log\bigl(\frac{\gamma}{1-\gamma}\bigr)/4$, whenever $\gamma\in(1/2,1)$ - see for example \cite{Dereich2012}. Since this condition on $\gamma$ also corresponds to the model being robust, one would naturally expect a similar correspondence between the parameter that leads to robustness and the scaling factor to exist in our model. Surprisingly, this is not the case.

As mentioned above Theorem \ref{Theorem}, robustness in our model emerges if there exists at least one index $\alpha_k$ such that the associated diameter is sufficiently heavy tailed relative to both the spatial dimension $d$ and the relative ordering of the diameter $k$ (times two). The natural assumption would therefore be that the scaling of the chemical distance depends on the largest or possibly smallest diameter that satisfies this requirement. Instead, the scaling is determined by the diameter that is, in relative terms (i.e.\ how $\alpha_k-k$ compares to $\min\{d-k,k\}$), the most heavy-tailed; see the definition of $\kappa$ in Theorem \ref{Theorem}. 
To understand how this comes to be, one has to remember that each ``step'' in a path requires two convex bodies to intersect. Consider for the sake of argument that $M=\{1,\dots,d-1\}$ and $\alpha_k-k=c\min\{d-k,k\}$ for some constant $c>0$ and all $k\in M$, that is, all $k\in M$ equal $\kappa$. Consider also two vertices whose locations are at a fixed large distance. Then, using basic trigonometry one can calculate the probabilities (in terms of random rotations) that two chosen diameters (one for each body) are roughly co-linear. When combined with the distributions of said two diameters and accounting for the combinatorics of different diameter pairings, one obtains (up to constants) that each diameter contributes equally to the probability of an intersection. If, however, one of the $d-1$ diameters (say $k$) is made even slightly more heavy-tailed, and consequently $\kappa\triangleq k$, as the distance between the two vertices tends to infinity, this diameter's contribution to the probability of an intersection asymptotically dominates all others.

We find that this unexpected behaviour, as well as the decoupling of the dense and robust regimes shown in \cite{Gracar2024a}, makes the Poisson-Boolean base model with regularly varying diameters uniquely interesting among it's Boolean model cousins, in particular since it exhibits these features without requiring the introduction of ``long edges'' (in the sense of long-range percolation, see for example \cite{Gracar2022c}).

\paragraph{Structure of this work.}
In Section~\ref{Sect:geometric} we prove several fairly general results that will be used to prove the bounds of the theorem. 
We prove the lower bound of the theorem in Section~\ref{Sect:lower}, by using \emph{truncated first moment method} similar to the one from \cite{Gracar2022a} in combination with an adapted version of the construction of infinite paths from \cite{Gracar2024a}. In addition to that, we comment on the behaviour of the chemical distance for the case $\alpha_k\geq \min\{2k,d\}$ for all $k\in\{1,\dots,d\}$ in Remark~\ref{rem:langsam}.
We then prove the upper bound in Section~\ref{Sect:upper} by adapting again the construction from \cite{Gracar2024a} and using a similar sprinkling argument as used in \cite{Gracar2022a}. We also comment on the case where there exists $k\in\{1,\dots,d\}$ such that $\alpha_k\leq k$ in Remark~\ref{rem:schnell}.
We conclude with some examples of models for which our result holds in Section~\ref{Sect:examples}. As we will see, our result extends the work of \cite{Hilario2021} to higher dimensions and refines their result for the case $d=2$.

Note that throughout the proofs we use $c\in(0,\infty)$ as a generic constant which may change its value at every inequality, but is always finite and may depend only on $d$, $\kappa$, $\alpha_{\kappa}$, $\epsilon$ and on $\varepsilon>0$ appearing in the Potter bounds.

\section{Geometric tools}\label{Sect:geometric}

We begin by providing several geometric results that will allow us to prove the claimed bounds. We first note that the smallest convex bodies that have diameters of size $D^{_{(1)}},\dots,D^{_{(d)}}$ are the convex hulls of the endpoints of all of the diameters, i.e.\ the $d$ dimensional convex polytopes with $2d$ vertices. We will, from here on out, refer to these endpoints as the polytope \emph{corners} in order to avoid ambiguity with the vertices of $\mathcal{X}$. We will also use corners to refer to the vertices of boxes for the same reason. In the following, we are interested in sufficiently large  boxes that can be contained in a convex body. Due to the above observation, it follows trivially that if we find such a box for a polytope, this box is also contained in any other convex set with the same diameters and orientation. Consequently, if an arbitrary set intersects this box, it must also intersect both the polytope and any other convex set containing the box. For this reason, we can without loss of generality work for the time being with polytopes only in order to keep the notation and proofs more concise.

\begin{lemma}\label{Lemma:one}
    Let $K\subset\mathbb{R}^d$ be an arbitrary convex polytope with $2d$ corners and diameters $0<l_d\leq\dots\leq l_1<\infty$. Then there exists a box $B:=B(K)\subset K$ which is congruent to 
    \begin{equation*}
        \bigtimes\limits_{i=1}^d [0,2^{-2(d-1)}l_i].
    \end{equation*}
\end{lemma}
\begin{proof}
    We consider first the case $d=2$. In this case, it is clear that the line segment of length $l_1$ connecting two of the polytope corners divides the polytope into two triangles. Both triangles have a hypotenuse of length $l_1$. We denote by $\tilde l_2$ (resp.\ $\hat l_2$) the height relative to the hypotenuse of the first (resp.\ second) triangle. The maximum of both is at least $\tfrac{l_2}{2}$ since together they equal $l_2$. Looking now at the bigger of the two triangles and using that triangles are convex, it is clear that there exists a box $B$ of side-lengths at least $\tfrac{l_2}{4}$ and $\tfrac{l_1}{2}$ that is contained in this triangle and therefore also in $K$ (see Figure \ref{fig:polytop_d2}). It can be quickly verified that if this is not possible, than $l_1$ would not have been the first diameter.
\begin{figure}[h!]
    \centering
    \includegraphics[width=1\textwidth]{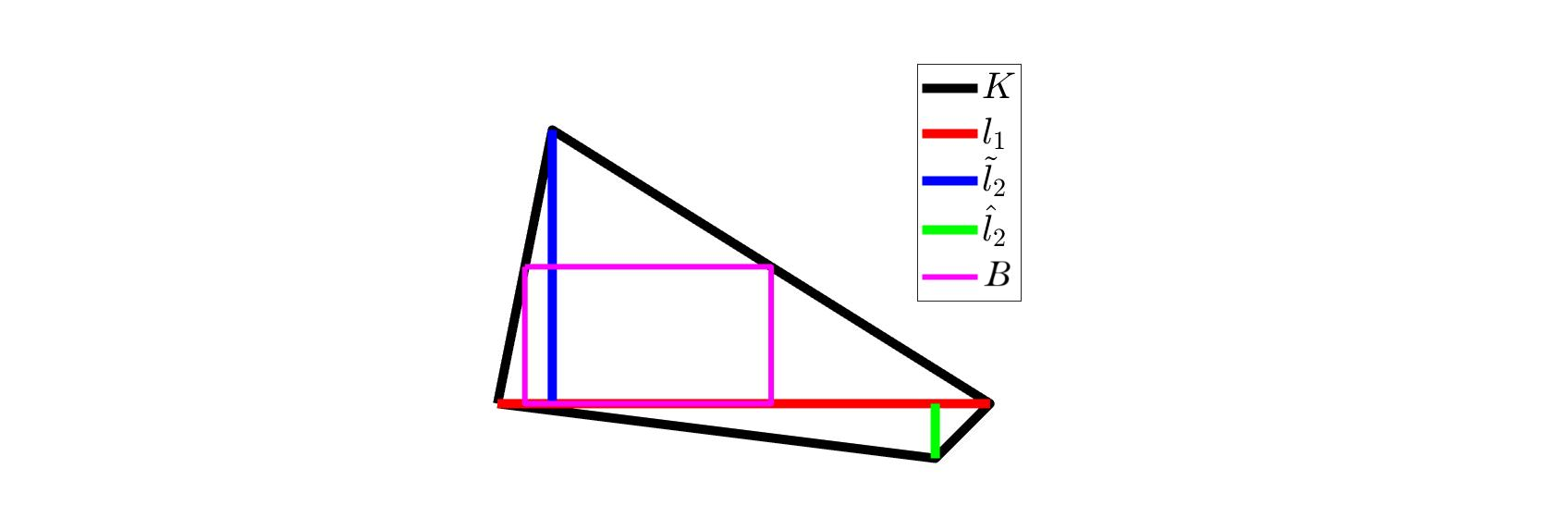} 
    \caption{$l_2=\tilde{l}_2+\hat{l}_2$. $\tilde{l}_2=\max\{\tilde{l}_2,\hat{l}_2\}\geq l_2/2$. $B$ with side-lengths at least $l_1/4$ and $l_2/4$.}
    \label{fig:polytop_d2}
\end{figure}

Assume now that the claim is true for a fixed integer $d\geq 2$. We consider now the induction step for $d+1$. Due to the rotation invariance of the volume in euclidean space, we assume that $l_1,\dots,l_{d+1}$ have orientations $e_1,\dots,e_{d+1}$ corresponding to the canonical basis of $\mathbb{R}^{d+1}$. We know that in the subspace spanned by the directions of the first $d$ diameters, i.e.\ spanned by $e_1,\dots,e_d$, there exists a $d$ dimensional box of side-lengths $2^{-2(d-1)}l_1,\dots,2^{-2(d-1)}l_d$ which we denote by $B_d$. Due to translation invariance, we assume without loss of generality that  $B_d=\bigtimes_{i=1}^d[-2^{-2d+1}l_i\,,\,2^{-2d+1}l_i]\,\times\{0\}$. 
Looking now at the two corners of the polytope which give rise to the $(d+1)$st diameter with length $l_{d+1}$, it is true that at least one of them has distance greater or equal $l_{d+1}/2$ from $\mathbb{R}^d\times\{0\}$, similar to the $d=2$ case above. We focus on the corner which satisfies this and denote it by $p$. Note that just like in the case $d=2$, $p$ has to lie ``above'' a $d$ dimensional box congruent to $\bigtimes_{i=1}^d [0,\ell_i]\times \{0\}$ that includes in its boundary the orthogonal projections of the first $2d$ corners of $K$ onto $\mathbb{R}^d\times\{0\}$, since otherwise at least for one $\ell_i$ for 
$i\in\{1,...,d\}$ could not be the $i$th diameter.
We next construct a box which is included in $K$ using that without loss of generality the $(d+1)$st coordinate of $p$ equals $l_{d+1}/2$.

Looking at the maximal distance of $B_d$ to $p$ in the direction $e_i$ for $i\in\{1,\dots,d\}$, it is clear that this distance is at most $(1-2^{-2(d-1)})l_i$, since $l_1,\dots,l_d$ are the sizes of the first $d$ diameters and $K$ is a polytope. A visualisation in 3 dimensions is given in Figure \ref{fig:distance_of_p}.
\begin{figure}[!h]
    \centering
    \begin{subfigure}[t]{\textwidth}
        \centering
        \begin{subfigure}[b]{0.49\textwidth}
            \centering
            \includegraphics[width=\textwidth]{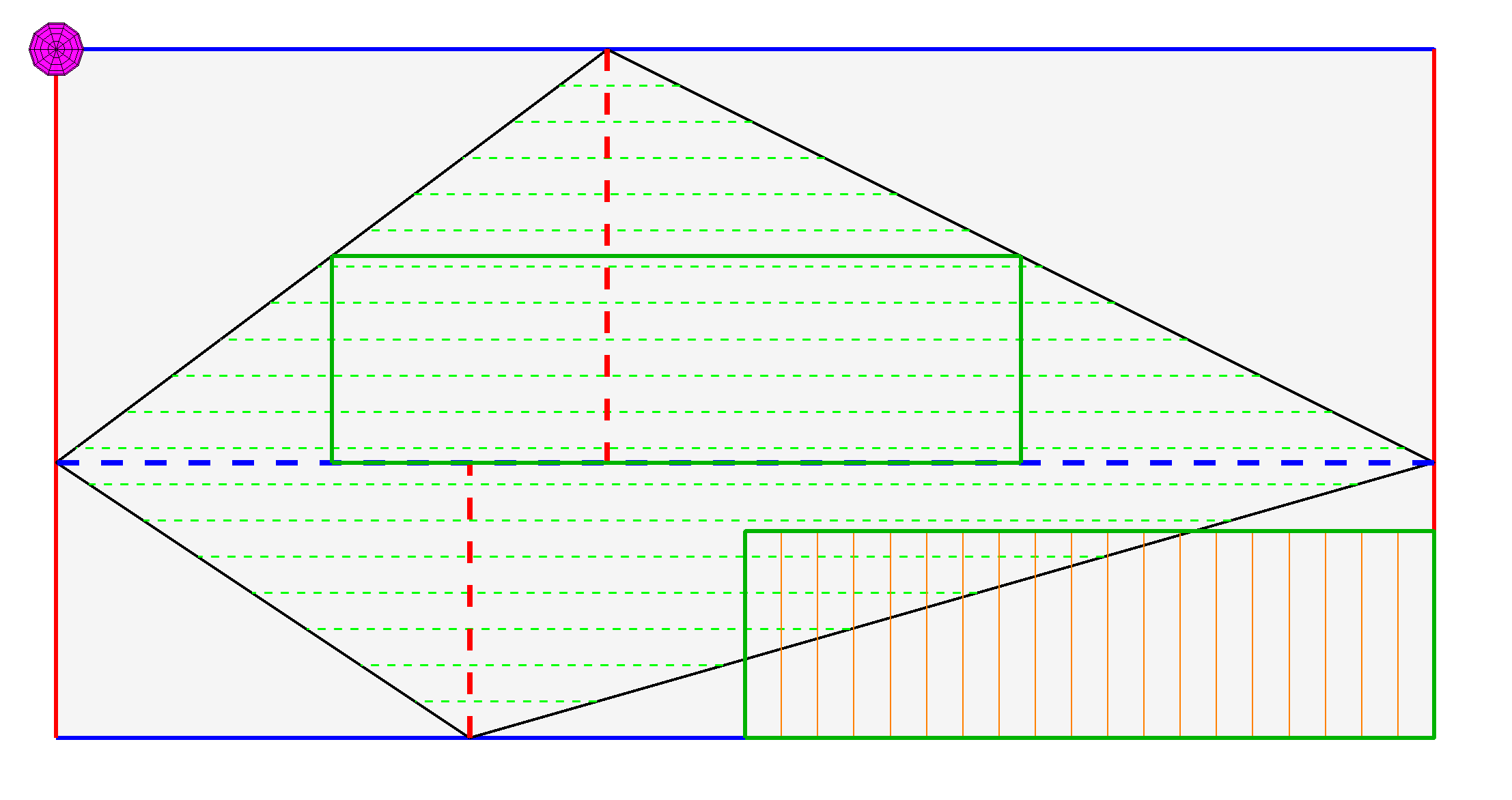}
            \caption{Perspective from above.}
        \end{subfigure}
        \begin{subfigure}[b]{0.49\textwidth}
            \centering
            \includegraphics[width=\textwidth]{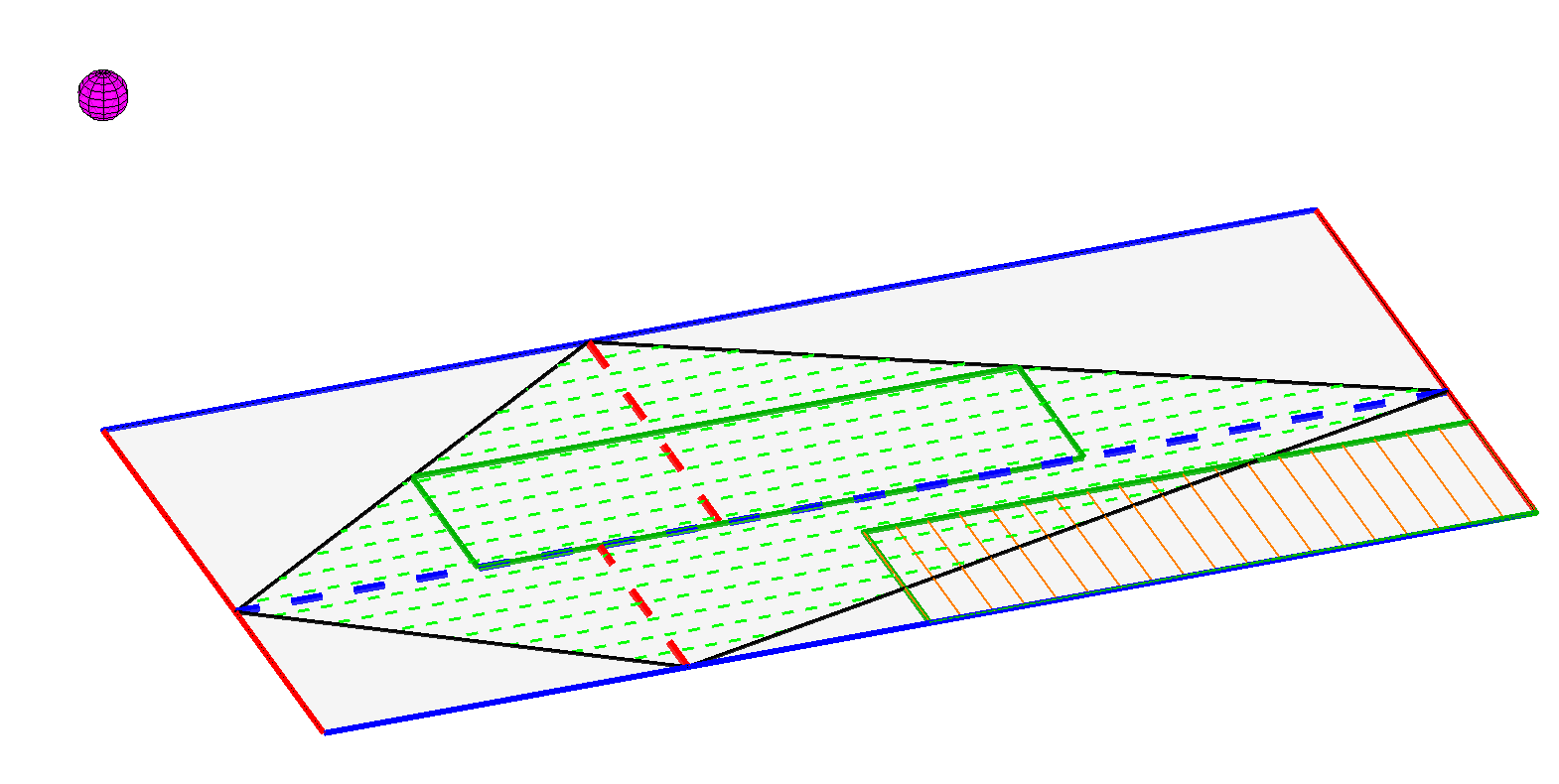}
            \caption{Perspective in 3 dimensions.}
        \end{subfigure}
    \end{subfigure}
    \caption{Visualisation of the maximal distance of $B_d$ and $p$ in the directions $e_i$ for $i\in\{1,\dots,d\}$ for the case $d=2$. The blue and red solid lines represent an upper bound for the polytope with diameters $l_1$ and $l_2$. The black lines and hatched green area denote the polytope given by the first $2d$ corners and the diameters of length $l_1,\dots,l_d$. The green rectangle in this convex set is $B_2$ while the rectangle with green boundary and hatched orange is $B_2$ shifted such that it has the maximal distance to $p$, which is given as the pink point. }
    \label{fig:distance_of_p}
\end{figure}
 We claim that a box which is congruent to \smash{$B:=\bigtimes_{i=1}^{d+1}[0,2^{-2d}l_i]$} exists inside of $K$. To find it, we have to consider two different cases. The first one is that $p^{(i)}\in[-2^{-2d+1}l_i,2^{-2d+1}l_i]$ for all $i\in\{1,\dots,d\}$, i.e.\ the orthogonal projection of $p$ onto $\mathbb{R}^d\times\{0\}$ is contained in $B_d$. Note that the convex hull of $B_d$ and $p$ is a hyperpyramid (see Figure \ref{fig:case1}). Considering now the connection lines of the corners of $B_d$ and $p$ and choosing the midpoints of these lines we obtain that the convex hull of these midpoints and their orthogonal projections onto $\mathbb{R}^d\times\{0\}$ determine a $(d+1)$ dimensional box, which we denote by $\tilde B$. This box is congruent to \smash{$\bigtimes_{i=1}^{d} [0,2^{-2d+1}l_i]\times[0,l_{d+1}/4]$} using the self-similarity property of hyperpyramids. Since this box contains a smaller box which is congruent to the sought after box $B$, we are done with this case. A visualisation of this step for $d+1=3$ is given in Figure~\ref{fig:case1}. 
\begin{figure}[!h]
    \centering
    \begin{subfigure}[t]{\textwidth}
        \centering
        \begin{subfigure}[b]{0.66\textwidth}
            \centering
            \includegraphics[width=\textwidth]{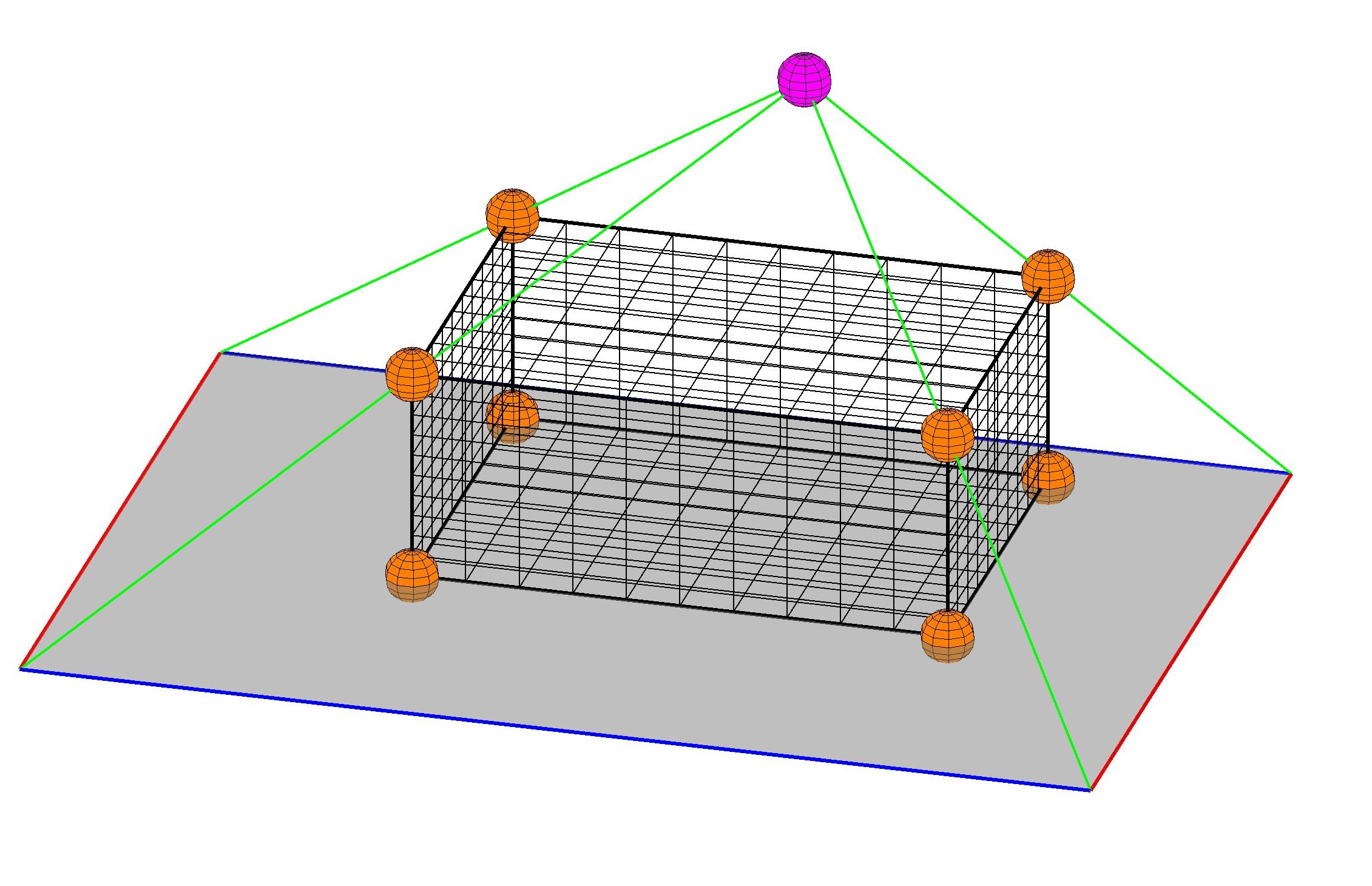}
            \caption{3 dimensional perspective.}
        \end{subfigure}
        \begin{subfigure}[b]{0.33\textwidth}
            \centering
            \includegraphics[width=\textwidth]{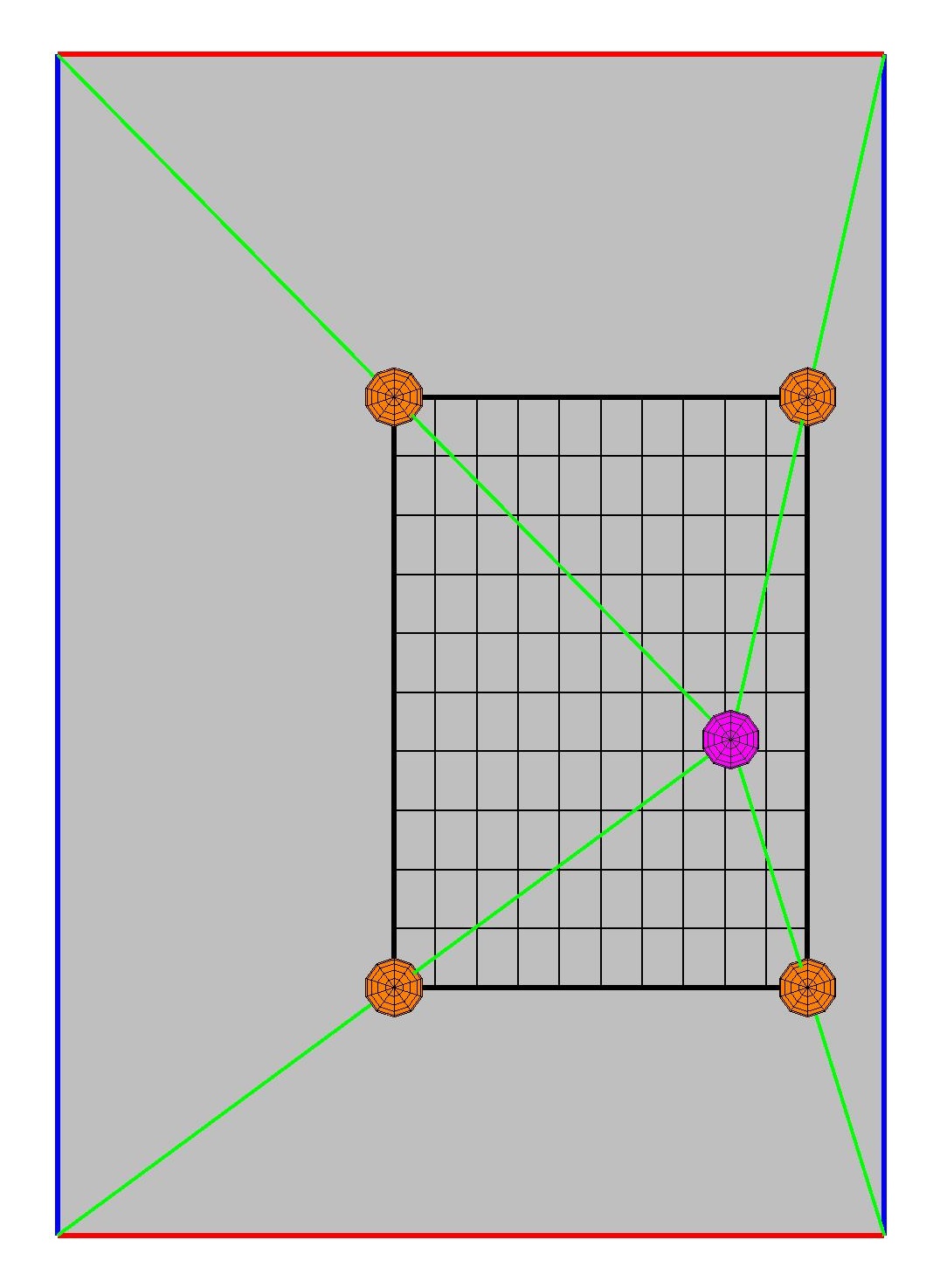}
            \caption{Perspective from above.}
        \end{subfigure}
    \end{subfigure}
    \caption{Example for the induction step for the first case of Lemma \ref{Lemma:one} from $d=2$ to $d+1=3$ from different perspectives. $B_2$ is in gray, $p$ is the pink point, in green are the connection lines of $p$ and the corners of $B_2$, the orange points as the midpoints of these lines and the orthogonal projection of the midpoints, and $\tilde B$ is the box represented by the black mesh.}
    \label{fig:case1}
\end{figure}

In the second case we are dealing with $p$ such that the orthogonal projection of it onto $\mathbb{R}^d\times\{0\}$ does not lie in $B_d$, see Figure \ref{fig:case2}. As we are interested in finding a box which is congruent to the sought after box $B$, it suffices to look for a minimal box and show that this is a suitable choice. We assume therefore as in the previous case, that $p$ has distance $l_{d+1}/2$ to $\mathbb{R}^d\times\{0\}$.

We consider now the worst case scenario for the location of $p$, i.e.\ that it has the maximal distance to the base $B_d$ in each direction $e_i$, $i\in\{1,\dots,d\}$, which is $(1-2^{-2(d-1)})l_i$ as discussed above. For that to be the case, $p$ has to lie in 
\begin{equation*}
    N:=\bigl\{x\in\mathbb{R}^{d+1}\,:\, (\pm(1-2^{-2d+1})l_1,\dots,\pm(1-2^{-2d+1})l_d, l_{d+1}/2) \bigr\}.
\end{equation*}
We will determine now a new vertex which will be used to construct the sought after box. 

Let $y\in B_d$ be such that it has maximal distance to $p$. The coordinates of such a vertex $y$ fulfil $y^{(i)}=\mp2^{-2d+1}l_i$ for all $i\in\{1,\dots,d\}$ and $y^{(d+1)}=0$. 
Due to the symmetry of $B_d$ we can without loss of generality look at $p=((1-2^{-2d+1})l_1,\dots,(1-2^{-2d+1})l_d, l_{d+1}/2)\in N$ and consequently set $y:=(-2^{-2d+1}l_1,\dots,-2^{-2d+1}l_d,0)\in B_d $. Consider now the point 
\begin{equation*}
    z:= 2^{-2(d-1)}(p-y)+y = (2^{-2d+1}l_1,\dots,2^{-2d+1}l_d,2^{-2d+1}l_{d+1})
\end{equation*}
which lies in the convex hull of $B_d$ and $p$, and in particular lies ``above'' $B_d$ (as indicated in the example of Figure~\ref{fig:case2}). We claim even more, namely that the convex hull of $z$ and $B_d$ is a minimal hyperpyramid which is also included in $K$, whereby it is minimal in the sense that its volume is minimal across for all choices of $p$. To see the latter, note that we are considering the worst case location of $p$, but that there exists an actual corner of the polytope that has the same distance from $\R^d\times\{0\}$ as $p$. If one were to construct the point $z$ (call it $\tilde z$) for this corner, one would immediately have that $\tilde z$ lies in $K$ and that the distance of $\tilde z$ to $\R^d\times\{0\}$ is not smaller than that of $z$ (since $z$ has the minimal possible distance for a vertex constructed in this way). 

With $z$ taking the role of $p$ we are back in the first case and using convexity and the midpoints of the connection lines from $z$ to the corners of $B_d$ gives us the desired box. The midpoints and their orthogonal projections are given by
\begin{equation*}
    \tilde{N}:=    \bigl\{x\in\mathbb{R}^{d+1}\,:\,x^{(i)}\in\{0,2^{-2d+1}l_i\}, \forall i\in\{1,\dots,d\}, x^{(d+1)}\in\{0,2^{-2d}\}\bigr\}.
\end{equation*}
The resulting box, denoted as $B$, has side-lengths $2^{-2d}l_1,\dots,2^{-2d}l_{d}$ which concludes the proof.

\begin{figure}[!h]
    \centering
    \begin{subfigure}[t]{\textwidth}
        \centering
        \begin{subfigure}[b]{0.66\textwidth}
            \centering
            \includegraphics[width=\textwidth]{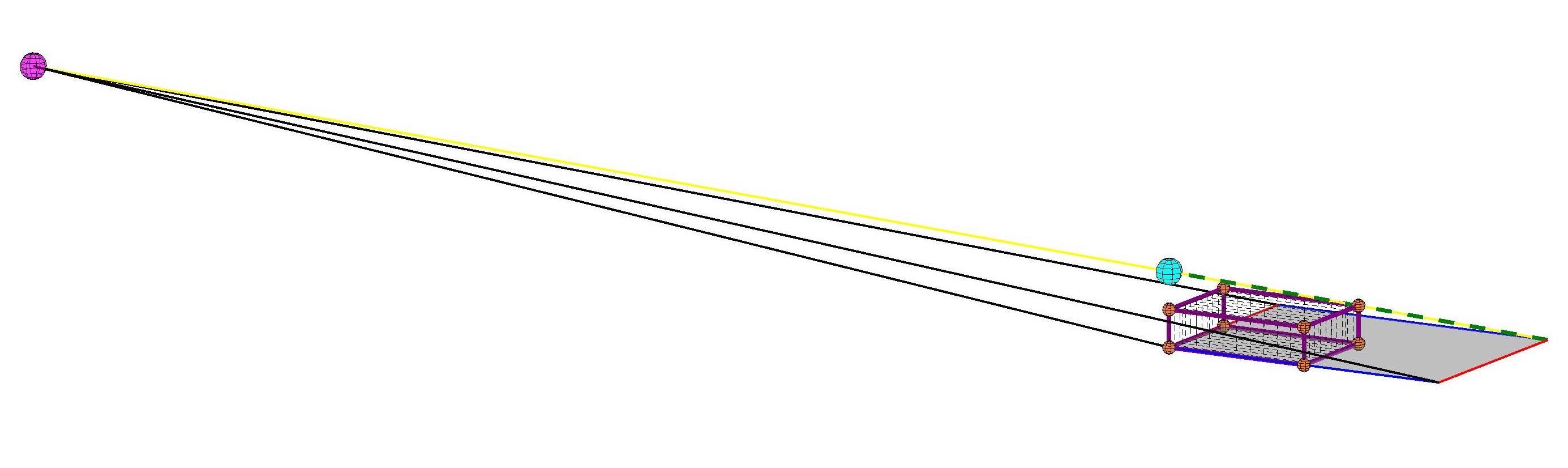}
        \end{subfigure}
        \begin{subfigure}[b]{0.33\textwidth}
            \centering
            \includegraphics[width=\textwidth]{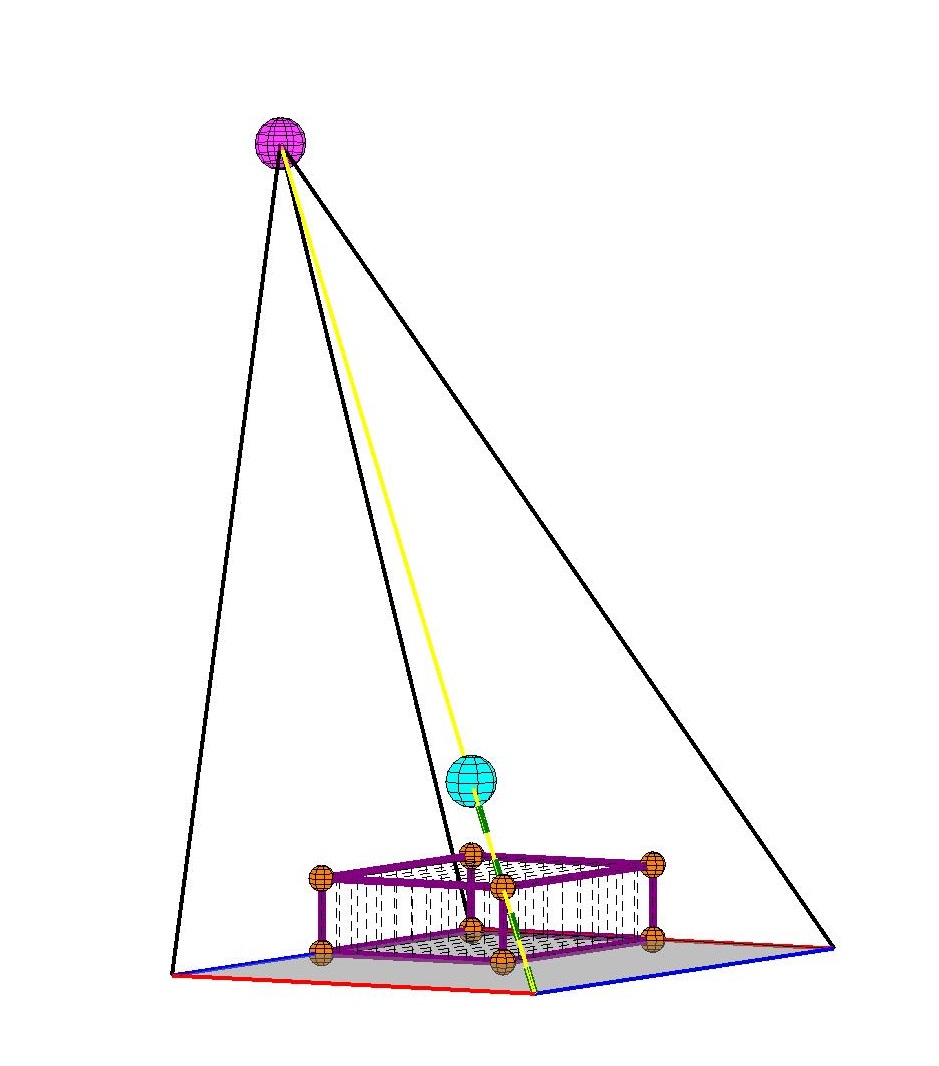}
        \end{subfigure}
        \caption{Induction step from $d=2$ to $d=3$ of the second case of $p$.}
    \end{subfigure}
    \begin{subfigure}[t]{\textwidth}
        \centering
        \begin{subfigure}[b]{0.66\textwidth}
            \centering
            \includegraphics[width=\textwidth]{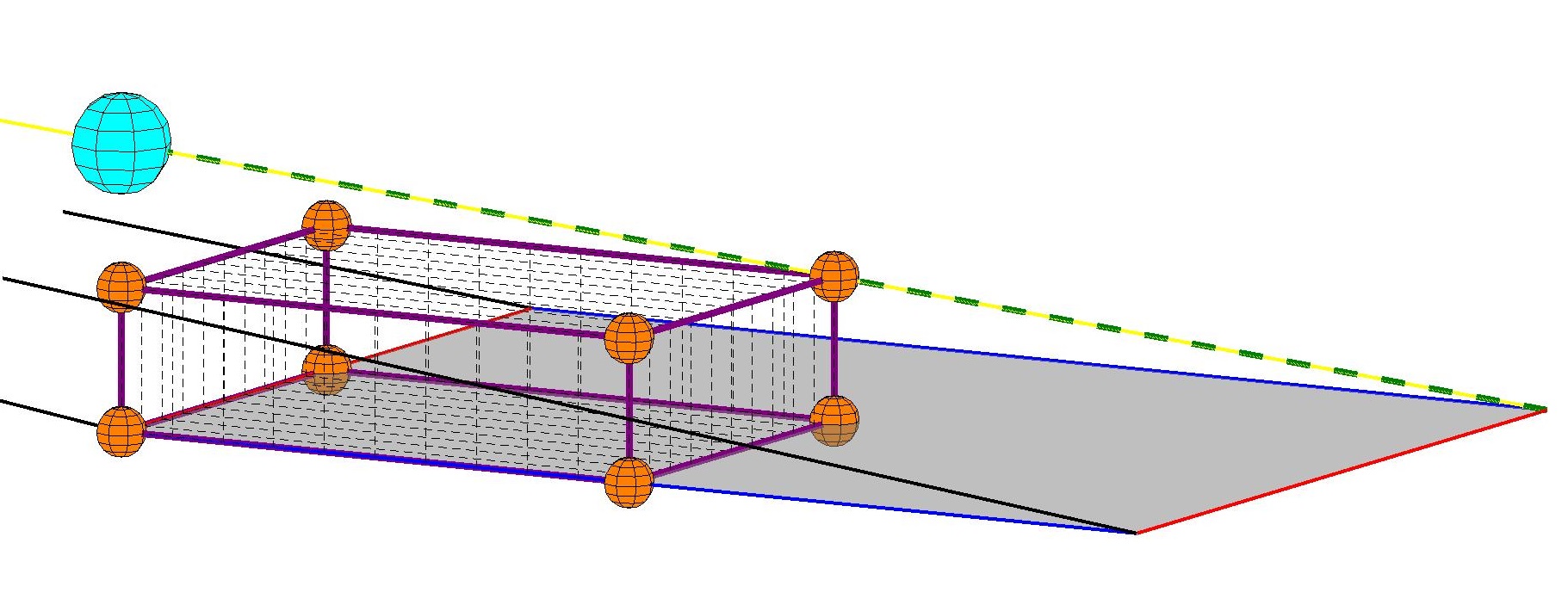}
        \end{subfigure}
        \begin{subfigure}[b]{0.33\textwidth}
            \centering
            \includegraphics[width=\textwidth]{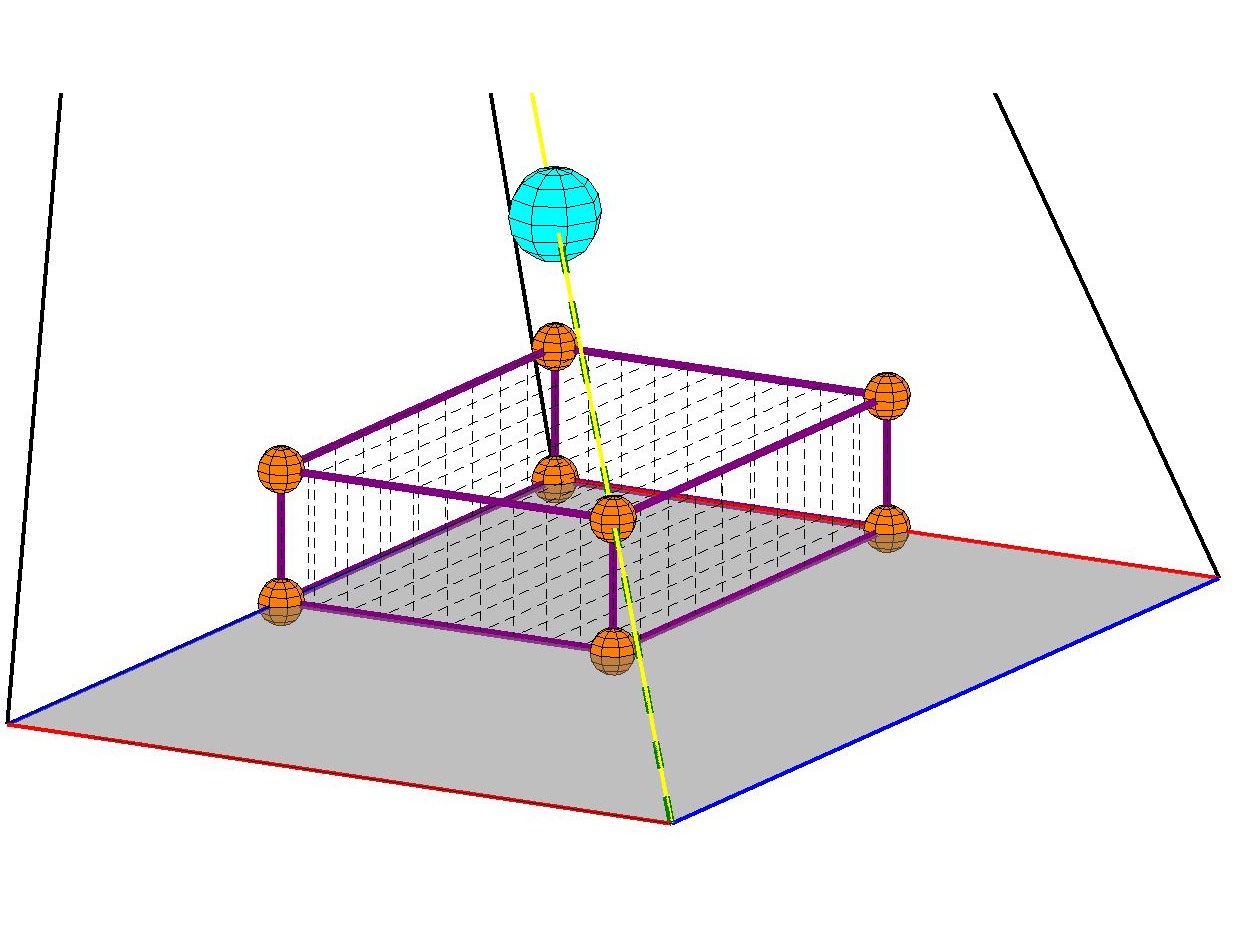}
        \end{subfigure}
        \caption{Induction step zoomed in from $d=2$ to $d=3$ of the second case for $p$.}
    \end{subfigure}
    \caption{Example for the described induction step of Lemma \ref{Lemma:one} from $d=2$ to $d+1=3$ of the second case for $p$. $B_2$ in gray, $p$ the pink point, in turquoise ${z=(2^{-2d+1}l_1,\dots,2^{-2d+1}l_d,2^{-2d+1}l_{d+1})}$, yellow the connection line of $p$ and $y$, in green the connection line of $z$ and $y$, the orange points as $\tilde{N}$ and in purple the boundary of $B$. The hyperpyramid is given as the convex hull of the turquoise point and $B_2$. The boundary of $B_2$ is indicated via the red and blue lines.}
    \label{fig:case2}
\end{figure}
We note that it is possible for one of the two hyperpyramids can be degenerate, that is, have $0$ volume. This however does not affect the proof, since we always work with the larger of the two when choosing the corner $p$.
\end{proof}

For $u>0$ we define $^{u}\mathbb{P}$ as the probability measure of the Poisson-Boolean base model where the underlying Poisson Point process for the locations of the vertices has intensity $u>0$. Let $C_{\infty}$ be an unbounded connected component of this model. Recall that by \cite{Chebunin2024}, if $C_{\infty}$ exists, it is almost surely unique.

We define the percolation probability
    \begin{equation*}
        \theta(u):= \,^{u}\mathbb{P}(0\in C_{\infty}).
    \end{equation*}

Following the discussion before Lemma \ref{Lemma:one}, we limit ourselves for the remainder of this section to the Boolean model where $C$ is given as a box with side-lengths $2^{-2(d-1)}D^{_{(1)}},\dots,2^{-2(d-1)}D^{_{(d)}}$. We can do this, since we will use the following two results only in the proof of the upper bound and since considering a ``smaller'' model can only increase the chemical distances. More precisely, 
the chemical distance for the model with boxes as above yields an upper bound for the general Poisson-Boolean base model with regularly varying diameters. In addition to that we assume also that the box $C$ has its rotation point and therefore also its location as its center, which we can do without loss of generally by using the mapping theorem, see e.g.\ in~\cite[Theorem~5.1]{Last2017}.
We denote in the following for $K\subset\mathbb{R}^d$ the interior of $K$ as $\interior(K)$ and the boundary of $K$ as $\partial K$. 
\begin{lemma}\label{Lemma:two}
For every $x\in\mathscr{P}$ we have \begin{equation*}
        \mathbb{P}_{0,x}(\{C_x\cap  \partial C_0\neq \emptyset\} \,\cap\, \{C_x\cap \interior(C_0)=\emptyset\})=0.
    \end{equation*}
\end{lemma}
\begin{proof}
Let $\mathbb{P}_{C_0,D_x}$ be the probability measure under the condition that $C_0$ and the size of the diameters of $C_x$ are fixed i.e.\ the only random part is the rotation of $C_x$.
We will argue the claim by using contradiction. 
Let $\rot_{C_0}$ be the rotation for which
\begin{equation*}
    C_0= \rot_{C_0}\bigl(\bigtimes_{i=1}^d [-2^{-2d+1}D^{_{(i)}}, 2^{-2d+1}D^{_{(i)}}]\bigr).
\end{equation*}
Define the $\delta$-reduction of $C_0$ as
\begin{equation*}
    C_0^{\delta}:=\rot_{C_0}\bigl(\bigtimes_{i=1}^d [-2^{-2d+1}D^{_{(i)}}+\delta, 2^{-2d+1}D^{_{(i)}}-\delta]\bigr),
\end{equation*}
and set $ C_0^{\delta}=\emptyset$ if $\delta>\epsilon2^{-2d+1}$. 
It holds that
\begin{equation*}
    \interior(C_0) = \bigcup\limits_{m\in\mathbb{N}} C_0^{1/m}.
\end{equation*}
Moreover we also have
\begin{equation*}
     \interior(C_0) = \bigcup\limits_{m\in\mathbb{N}} \interior(C_0^{1/m}).
\end{equation*}
Furthermore, it holds that
\begin{equation*}
     \partial C_0 = \lim_{m\rightarrow\infty} \partial C_0^{1/m}.
\end{equation*}
Assume now, incorrectly, that 
\begin{equation}\label{eq:contradiction}
        \mathbb{P}_{C_0,D_x}(\{C_x\cap  \partial C_0\neq \emptyset\} \,\cap\, \{C_x\cap \interior(C_0)=\emptyset\})>0.
\end{equation}

Using the preceding observations, we can obtain 
\begin{align*}
    \mathbb{P}_{C_0,D_x}(C_x \cap \interior(C_0)\neq \emptyset)& = \lim\limits_{m\rightarrow\infty} \mathbb{P}_{C_0,D_x} (C_x \cap C_0^{1/m}\neq \emptyset)\\
    &=\lim\limits_{m\rightarrow\infty} \mathbb{P}_{C_0,D_x} (C_x \cap \interior(C_0^{1/m})\neq \emptyset) \\
    &\hspace{20pt}+ \lim\limits_{m\rightarrow\infty} \mathbb{P}_{C_0,D_x} (\{C_x \cap \partial C_0^{1/m}\neq \emptyset\} \cap\{C_x \cap \interior(C_0^{1/m})= \emptyset\} )\\
    &>\mathbb{P}_{C_0,D_x}(C_x \cap \interior(C_0)\neq \emptyset),
\end{align*}
where the inequality follows from \eqref{eq:contradiction} and the third observation above. Since this yields
\begin{equation*}
\mathbb{P}_{C_0,D_x}(C_x \cap \interior(C_0)\neq \emptyset)> \mathbb{P}_{C_0,D_x}(C_x \cap \interior(C_0)\neq \emptyset),
\end{equation*}
which is a contradiction, we must have that
\begin{equation*}
    \mathbb{P}_{C_0,D_x}(\{C_x\cap  \partial C_0\neq \emptyset\} \,\cap\, \{C_x\cap \interior(C_0)=\emptyset\})=0.
\end{equation*}
Integrating this over all possible $C_0$ and $D_x$ leads then to the stated claim.
\end{proof}

\begin{proposition}\label{Prop:one}
    The function $u\mapsto \theta(u)$ is continuous.   
\end{proposition}
\begin{proof}
    The proof can be adapted to the proof with balls as convex grains as in \cite[Theorem 3.9]{Sarkar1997} or \cite{Meester1996a}. The only thing we have to check is $\mathbb{P}_{0,x}(\{C_x\cap \partial C_0\neq\emptyset\}\,\cap\,\{C_x\cap \interior(C_0)=\emptyset\})=0$ for all $x\in\mathscr{P}$, which we proved in Lemma \ref{Lemma:two}. 
\end{proof}
\section{Proof of the lower bound}\label{Sect:lower}
In this section we will prove the lower bound for the chemical distance. For this and also for the upper bound we will use a construction of infinite paths similar to the one introduced in \cite{Gracar2024a}. 

We first consider the case $M\neq \emptyset$ with $\alpha_k>k$ for all $k\in\{1,\dots,d\}$ and start with an increasing threshold sequence, which we denote by $(f_n)_{n\in\mathbb{N}}$ and define as
\begin{equation*}
      f_n:=\bigl(f_{n-1}\bigr)^{\frac{\min\{d-\kappa,\kappa\}}{\alpha_\kappa-\kappa}-\epsilon},\, \mbox{ for } n\in\mathbb{N}.
\end{equation*} 
Here, $f_0>1$ can be chosen arbitrarily big and $\kappa$ is defined as in Theorem \ref{Theorem}. Note that the $\epsilon$ in the definition of $f_n$ is the same as previously defined in Section \ref{Sect:Introduction} and without loss of generality we assume further that $0<\epsilon<\tfrac{1}{4}$; this ensures that the exponent in the above sequence is strictly bigger than~1 and makes the sequence increasing. As $\kappa\in M$ this is possible since $\alpha_{\kappa}< \min\{2\kappa,d\}$. 

We use the above threshold sequence to define two very similar sequences of events that our infinite path will have to satisfy. We will use the first of these event sequences in this proof and use the other for the proof of the upper bound. Due to the similarity of the relevant events however, we introduce both of them here to keep unnecessary repetition to a minimum. To define the two sequences we, as in \cite{Gracar2024a}, define the following two sets.
For a vertex $\mathbf{x}\in\mathcal{X}$ we define     
        \begin{equation}\label{eq:Oi}
            O_{_{i}}(\mathbf{x}):= \Bigl\{y\in\mathbb{R}^d: |x-y| \in [\tfrac12 f_{_{i}},f_{_{i}}], \measuredangle\big(x-y, \pm p_{x}^{_{(j)}}\big)> \varphi \mbox{ for all }  1\leq j\leq \kappa \Bigr\},
        \end{equation}
with $\varphi=2^{-\kappa}\pi$ and $\measuredangle(x,y)$ as the angle between vectors $x,y\in\mathbb{R}^d$. $\pm$ means that the condition on the angles is supposed to hold for the angle between $x-y$ and $+p_{x}^{_{(j)}}$ as well as $-p_{x}^{_{(j)}}$. 
Defining the ball $B_r(z):=\{w\in\mathbb{R}^d:\,|w-z|\leq r\}$ for $r>0$, $z\in\mathbb{R}^d$, we define for $\mathbf{x}\in\mathcal{X}$ and $y\in\mathbb{R}^d$
        \begin{equation} \label{eq:B_stern}
            B^{*}_{f_{_{i-1}}}(\mathbf{x},y) = P_{H_{x-y}}\Big(C_{x}\cap B_{f_{_{i-1}}}(x)\Big),
        \end{equation}
where $H_{x-y}$ is the hyperplane perpendicular to $x-y$ with \smash{$|H_{x-y} \cap \partial B_{f_{_{i-1}}}(x)| =1$} and \smash{$\dist(H_{x-y}, x)<\dist(H_{x-y},y)$}, i.e. as the hyperplane that touches the ball of radius~$f_{_{i-1}}$ with $x$ as the centre, and $P_A$ is the orthogonal projection onto a hyperplane $A\subset \mathbb{R}^d$.

Using this we define for $x\in\mathscr{P}$ the two afore mentioned sequences of events $(A_n^x)_{n\in\mathbb{N}}$ and $(\bar A_n^x)_{n\in\mathbb{N}}$ as 
\begin{align*}
            A_n^{x}  := \Bigg\{ \begin{split}\exists &\mathbf{x}_1,\ldots, \mathbf{x}_n\in\mathcal{X}\colon \mathbf{x}_i\neq\mathbf{x}_m \text{ for all } i\neq m, D_{x_i}^{_{(\kappa)}}\geq 2^{2(d-1){+\epsilon}}f_{_{i}}, 
            D_{x_i}^{_{(1)}}\leq 2^{2(d-1)+\epsilon}f_{_{i}}^{\frac{2\alpha_{\kappa}}{\alpha_1}},\\& x_i\in O_i(\mathbf{x}_{i-1})\text{ and }C_{x_i}\cap B^{*}_{f_{_{i-1}}}(\mathbf{x}_{i-1},x_i)\neq \emptyset \text{ for all } 1\leq i\leq n, \text{ with } x_0=x. \end{split}\Bigg\}
\end{align*}
and 
\begin{align*}
            \bar A_n^{x}  := \Bigg\{ \begin{split}\exists &\mathbf{x}_1,\ldots, \mathbf{x}_n\in\mathcal{X}\colon \mathbf{x}_i\neq\mathbf{x}_m \text{ for all } i\neq m, D_{x_i}^{_{(\kappa)}}\geq 2^{2(d-1){+\epsilon}}f_{_{i}}, 
            \\& x_i\in O_i(\mathbf{x}_{i-1})\text{ and }C_{x_i}\cap B^{*}_{f_{_{i-1}}}(\mathbf{x}_{i-1},x_i)\neq \emptyset \text{ for all } 1\leq i\leq n, \text{ with } x_0=x. \end{split}\Bigg\}.
\end{align*}
$(A_n^x)_{n\in\mathbb{N}}$ will be used to construct the paths in the current proof for the lower bound of the chemical distance, while $(\bar A_n^x)_{n\in\mathbb{N}}$ will play a similar role for the upper bound. We note that latter sequence matches the definition $(A_n)_{n\in\mathbb{N}}$ from \cite{Gracar2024a}\footnote{A reader familiar with \cite{Gracar2024a} might notice the additional factor $2^{\epsilon}$ that is not present in the other work. This factor should be present there as well, but has no effect on the computations and conclusions presented there.}; unlike in that work, we highlight here the dependence of the events on the starting vertex $\mathbf{x}$ in the superscript. 

\begin{rem}\label{rem:A_bar_kappa}
    Note that the definitions of $A_n^{x}$ and $\bar A_n^{x}$ both depend on $\kappa$. Later on, in Section \ref{Sect:upper}, when we will be working with $\bar A_n^{x}$, we will replace $\kappa$ with some fixed $k\in M$. This change does not affect any of the calculations that relate to $\bar A_n^{x}$ we do in this section. However, in order to avoid writing all calculations twice, or repeatedly pointing out that $\kappa$ should be replaced by $k$ when working with $\bar A_n^{x}$, we allow ourselves this small abuse of notation and simply write $\kappa$ overall. We will remind the reader when $\kappa$ is to be replaced with $k$ at the appropriate step in Section \ref{Sect:upper}.
\end{rem}

Roughly speaking, $A_n^x$ and $\bar A_n^x$ are the events that there exists a path of length $n$ which starts in $\mathbf{x}\in\mathcal{X}$ and satisfy for every $i\in\{1,\dots,n\}$ the following.

        \begin{itemize}
            \item The first $\kappa$ diameters of $\mathbf{x}_i$ do not fall below the threshold $2^{2(d-1){+\epsilon}}f_{_{i}}$. For $A_n^x$ we have additionally that the first diameter of $\mathbf{x}_i$ is bounded from above by $2^{2(d-1)+\epsilon}f_n^{_{\frac{2\alpha_{\kappa}}{\alpha_1}}}$. The latter property is necessary for the proof in the lower bound to get an upper bound for the distance that can be covered by the first $n$ vertices of a path, and is also the only difference from the other sequence.
            \item $x_i\in O_i(\mathbf{x}_{i-1})$ ensures that $x_i$ does not lie ``too close'' to the affine subspace through $x_{i-1}$ spanned by the orientations 
            of the first $\kappa$ big diameters of $\mathbf{x}_{i-1}$.  More precisely, we require the angle between any spanning vector of this subspace through $x_{i-1}$ and the vector $x_{i}-x_{i-1}$ 
            to be larger than some $\varphi\in\mathbb{S}^{d-1}$. 
            \item $x_i\in O_i(\mathbf{x}_{i-1})$ further ensures that the distance between points $x_{i-1}$ and $x_i$ is at least  $\frac12 f_{_{i}}$ and at most $f_{_{i}}$, and $x_i\in B_{2f_{_{i}}}(0)$. 
            \item $C_{x_i}\cap B^{*}_{f_{_{i-1}}}(\mathbf{x}_{i-1},x_i)\neq \emptyset$ ensures that the convex body $C_{x_i}$ intersects a special\footnote{This part of the hyperplane is defined in such a way that if $C_{x_i}$ intersects it, it must also intersect~$C_{x_{i-1}}$.} part of a hyperplane ``behind'' the convex body~$C_{x_{i-1}}$. 
        \end{itemize}

Using $(A_n^x)_{n\in\mathbb{N}}$ combined with the \emph{truncated first moment method} of \cite{Gracar2022a} will give us the lower bound for the chemical distance. The idea of the truncated first moment method is to choose $\Delta\in\mathbb{N}$ big enough such that
$\mathbb{P}_{x,y}(\dist(\mathbf{x},\mathbf{y}) \leq 2\Delta)$ is arbitrarily small. Defining $G_n^{(x,y)}$ as the event that a \emph{good} (in some yet to be determined sense) path of length $n$ connecting $\mathbf{x},\mathbf{y}$ exists, and $B_n^{(x)}$ as the event that a \emph{bad} (i.e.\ not good) path starting in $\mathbf{x}$ of length $n$ exists, we have by counting the numbers of bad and good paths of length $\Delta$, respectively $2\Delta$, the following inequality:
\begin{equation*}
    \mathbb{P}_{x,y}(\dist(\mathbf{x},\mathbf{y})\leq 2\Delta) \leq \sum\limits_{n=1}^{\Delta} \mathbb{P}_{x}(B_n^{(x)}) +\sum\limits_{n=1}^{\Delta}\mathbb{P}_{y}(B_n^{(y)}) +\sum\limits_{n=1}^{2\Delta} \mathbb{P}_{x,y}(G_n^{(x,y)}).
\end{equation*}
We now formalise what a good path is. We say that a path of length 
$n$ connecting $\mathbf{x}$ and $\mathbf{y}$ via vertices $\mathbf{x}_1,\dots,\mathbf{x}_{n-1}\in\mathcal{X}$ such that with $\mathbf{x}_0=\mathbf{x}$ and $\mathbf{x}_n=\mathbf{y}$ is good, 
if the vertices $\mathbf{x}_0,\dots,\mathbf{x}_{\lceil n/2\rceil}$ imply that $A_{\lceil n/2\rceil}^x$ holds and $\mathbf{x}_n,\dots,\mathbf{x}_{\lceil n/2\rceil}$ imply  that $A_{\lceil n/2\rceil}^y$ holds.
For $n\in\mathbb{N}$ and $\mathbf{x}\in\mathcal{X}$ we say that a path consisting of vertices $\mathbf{x}_0,\dots,\mathbf{x}_{n}$ is bad, if the vertices $\mathbf{x}_0,\dots,\mathbf{x}_{n-1}$ imply that $A_{n-1}^x$ holds, but $\mathbf{x}_0,\dots,\mathbf{x}_{n}$ do not imply that $A_n^x$ holds.

Consider now a good path of length $n$ and note that the property of being good imposes a maximum euclidean distance such a path can reach starting from a location $x\in\mathscr{P}$. Setting $\Delta$ small enough so that the sum over the maximum lengths of the first diameters of the $n$ convex bodies in this path is smaller than $|x-y|$, we get that $\sum_{n=1}^{\Delta} \mathbb{P}_{x,y}(G_n^{(x,y)})=0.$ By definition, a good path of length $n$ satisfies $D^{_{(1)}}_{x_i} < 2^{2(d-1)+\epsilon}f_{_{i}} ^{_{\frac{2\alpha_{\kappa}}{\alpha_1}}}$ for $i\in\{0,\dots,\lceil n/2\rceil\}$ and $D^{_{(1)}}_{x_{n-i+1}} < 2^{2(d-1)+\epsilon}f_{_{i}} ^{_{\frac{2\alpha_{\kappa}}{\alpha_1}}}$ for $i\in\{1,\dots,\lceil n/2\rceil\}.$ We therefore have to find the maximal even $m \in\mathbb{N}$ such that
\begin{equation*}
    2\sum\limits_{n=1}^{m/2} 2^{2(d-1)+\epsilon}f_n^{\frac{2\alpha_{\kappa}}{\alpha_1}} < |x-y|,
\end{equation*}
which will give us the largest possible value for $\Delta$.
Using the definition of the threshold sequence leads to the following inequality
\begin{equation}\label{eq:TMM_Delta_1}
    2\sum\limits_{n=1}^{m/2} 2^{2(d-1)+\epsilon}f_0^{\frac{2\alpha_{\kappa}}{\alpha_1}\left(\frac{\min\{d-\kappa,\kappa\}}{\alpha_\kappa-\kappa}-\epsilon\right)^n} < |x-y|.
\end{equation}
Choosing $f_0$ big enough it follows that 
\begin{equation*}
    2\sum\limits_{n=1}^{m/2} 2^{2(d-1)+\epsilon}f_0^{\frac{2\alpha_{\kappa}}{\alpha_1}\left(\frac{\min\{d-\kappa,\kappa\}}{\alpha_\kappa-\kappa}-\epsilon\right)^n} \leq 3\cdot 2^{2(d-1)+\epsilon}f_0^{\frac{2\alpha_{\kappa}}{\alpha_1}\left(\frac{\min\{d-\kappa,\kappa\}}{\alpha_\kappa-\kappa}-\epsilon\right)^{m/2}}.
\end{equation*}
With that we have that \eqref{eq:TMM_Delta_1} is true if 
\begin{equation*}
    3\cdot 2^{2(d-1)+\epsilon}f_0^{\frac{2\alpha_{\kappa}}{\alpha_1}\left(\frac{\min\{d-\kappa,\kappa\}}{\alpha_\kappa-\kappa}-\epsilon\right)^{m/2}} < |x-y|.
\end{equation*}
Applying a double logarithm on both sides of the inequality we see that this is equivalent to 
\begin{equation*}
    \frac{m}{2}\log\bigl(\frac{\min\{d-\kappa,\kappa\}}{\alpha_\kappa-\kappa}-\epsilon\bigr) + \log\Bigl(\frac{2\alpha_{\kappa}}{\alpha_1}\log(f_0)+\log(3)+\log\bigl(2^{2(d-1)+\epsilon}\bigr)\Bigr) < \log\log |x-y|.
\end{equation*}
This leads to the following upper bound for $m$:
\begin{equation}\label{eq:TMM_Delta_2}
    m<  2 \frac{\log\log |x-y|- \log\Bigl(\frac{2\alpha_{\kappa}}{\alpha_1}\log(f_0)+\log(3)+\log\bigl(2^{2(d-1)+\epsilon}\bigr)\Bigr)} {\log\bigl(\frac{\min\{d-\kappa,\kappa\}}{\alpha_\kappa-\kappa}-\epsilon\bigr)}.
\end{equation}
Let $\delta>0$. If we choose $m$ small enough to satisfy 
\begin{equation*}
    m \leq\frac{(2-\delta)\log\log|x-y|}{\log\bigl(\frac{\min\{d-\kappa,\kappa\}}{\alpha_\kappa-\kappa}-\epsilon\bigr)},
\end{equation*}
we get that as $|x-y|\rightarrow \infty$, such an $m$ also satisfies inequality \eqref{eq:TMM_Delta_2} and therefore also \eqref{eq:TMM_Delta_1}.
Choosing now 
\begin{equation*}
    2\Delta = \frac{(2-\delta)\log\log|x-y|}{\log\bigl(\frac{\min\{d-\kappa,\kappa\}}{\alpha_\kappa-\kappa}-\epsilon\bigr)}
\end{equation*}
therefore yields $\sum_{n=1}^{2\Delta} \mathbb{P}_{x,y}(G_n^{(x,y)})=0$ for $|x-y|$ large enough, as desired.

Looking now at the bad paths we will show that
\begin{equation*}
    \sum\limits_{n=1}^{\infty}\mathbb{P}_{x}(B_n^{(x)}) < \delta(f_0),
\end{equation*}
where $\delta$ is a function satisfying $\lim_{t\to\infty}\delta(t)=0$. 
We have by definition that  
\begin{equation*}
    \mathbb{P}_x(B_n^{(x)}) = \mathbb{P}_x\bigl((A_n^x)^c \cap A_{n-1}^x\bigr) \leq  \mathbb{P}_x\bigl((A_n^x)^c\,|\,A_{n-1}^x\bigr).
\end{equation*}
Finding a suitable summable upper bound for $\mathbb{P}_x((A_n^x)^c\,|\,A_{n-1}^x)$ for all $n\in\mathbb{N}$ will complete the proof. For this we adapt a calculation from \cite[Section 2.2]{Gracar2024a} and highlight here only the relevant changes. Note that these changes are not necessary to prove the same statements for the events $\bar A_n^{x}$ which we will use in the proof of the upper bound. To continue, we need the following definitions.         
Let $\rot_{\vartheta}$ be a rotation such that $\rot_{\vartheta}(e_1)=\vartheta$. Note that this rotation is not unique and we can fix one arbitrarily. In addition to that we define for $x,y\in\mathbb{R}^d$ $\rho_{x,y}:=\tfrac{x-y}{|x-y|}\in \mathbb{S}^{d-1}$ as the orientation of the vector $x-y$.  With that we define
\begin{align*}
    Q_{f_n}(x,y) := \rot_{\rho_{x,y}}\Bigl(\operatorname{conv}\bigl(&\bigl\{-\tfrac{\epsilon}{2}e_i,\tfrac{\epsilon}{2}e_i\,: \, 1\leq i \leq d-\kappa\bigr\}\\ &\cup \bigl\{-\tfrac{f_n}{2}e_m,\tfrac{f_n}{2}e_m\, : \, d-\kappa+1\leq m\leq d\bigr\}\bigr)\Bigr) +x
\end{align*}
where $\operatorname{conv}(A)$ is the convex hull of the set $A$ and define
\begin{equation*}
    Q^*_{_{f_n}}(x,y):= P_{H_{x-y}}\bigl(Q_{f_n}(x,y)\bigr).
\end{equation*}
with $P_{H_{x-y}}$ as the same orthogonal projection as in \eqref{eq:B_stern}. Roughly speaking, $Q^*_{_{f_n}}$ is an approximation of the set $B^{*}_{f_{_{i-1}}}$. In addition to that recall the Definition \ref{Def:model} for $\mathbb{P}_C$ as the law of $C$. 

The key step is to get a suitable lower bound for 
\begin{equation*}
    \mathbb{P}_C\bigl(  (x+C)\cap   Q^*_{f_n}(x_n,x)\neq \emptyset\,|\, D_{C}^{_{(\kappa)}} \geq 2^{2(d-1)+\epsilon}f_{_{n+1}} \bigr)  \mathbb{P}_C(D_{C}^{_{(\kappa)}} \!\!\geq 2^{2(d-1)+\epsilon}f_{_{n+1}}).
\end{equation*}
Following the same arguments as in \cite[Section 2.2]{Gracar2024a}, this is equivalent to finding a lower bound for
\begin{equation}\label{eq:TMM_key_step}
    \mathbb{P}_C\Biggl(  \!\!(x+C)\cap   Q^*_{f_n}(x_n,x)\neq \emptyset \Bigg|\!\! 
    \begin{array}{c}
        D_{C}^{_{(\kappa)}} \geq 2^{2(d-1){+\epsilon}}f_{_{n+1}},  \\
        D_{C}^{_{(1)}} \leq 2^{2(d-1)+\epsilon}f_{_{n+1}}^{\frac{2\alpha_{\kappa}}{\alpha_1}}
    \end{array}\!
    \Biggr) \mathbb{P}_C\Biggl(\!\!\!
    \begin{array}{c}
        D_{C}^{_{(\kappa)}} \geq 2^{2(d-1){+\epsilon}}f_{_{n+1}},  \\
        D_{C}^{_{(1)}} \leq 2^{2(d-1)+\epsilon}f_{_{n+1}}^{\frac{2\alpha_{\kappa}}{\alpha_1}}
    \end{array}\!\!
    \Biggr).
\end{equation}
The conditional probability depends only on the volume of the set of suitable orientations for $C$ that result in an intersection of $(x+C)$ with $Q^*_{f_n}(x_n,x)$. As in \cite{Gracar2024a} we get 
\begin{equation*}
    \mathbb{P}_C\Biggl(  (x+C)\cap   Q^*_{f_n}(x_n,x)\neq \emptyset \,\Bigg|\, 
    \begin{array}{c}
        D_{C}^{_{(\kappa)}} \geq 2^{2(d-1){+\epsilon}}f_{_{n+1}},  \\
        D_{C}^{_{(1)}} \leq 2^{2(d-1)+\epsilon}f_{_{n+1}}^{\frac{2\alpha_{\kappa}}{\alpha_1}}
    \end{array}
    \Biggr)
    \geq c \frac{f_n^{\min\{d-\kappa,\kappa\}}}{f_{_{n+1}}^{d-\kappa}}.
\end{equation*}
For the second probability term in \eqref{eq:TMM_key_step} we have to use the Potter bounds for regularly varying random variables (see \cite[Prop.\ 1.4.1]{Kulik2020}). This gives us for $X$ a regularly varying random variable with index $-\alpha$, that for every $\varepsilon\in(0,\alpha)$ there exists $c(\varepsilon)>0$ such that for $x\geq\varepsilon$
\begin{equation*}
    c(\varepsilon)^{-1}\, x^{-\alpha -\varepsilon} \leq \mathbb{P}(X >x)\leq c(\varepsilon)\,x^{-\alpha +\varepsilon}.
\end{equation*}

We therefore have for sufficient small $\varepsilon$ that for $s>\varepsilon $ and $i\in\{1,\dots,d\}$
\begin{equation*}
    c(\varepsilon)^{-1}s^{-\alpha_{i}-\varepsilon} \leq \mathbb{P}_C(D_C^{(i)} \geq s) \leq c(\varepsilon)s^{-\alpha_{i}+\varepsilon}.
\end{equation*}
Using this we have for $s>\varepsilon$
\begin{align*}
    \mathbb{P}_C\bigl(D_C^{(\kappa)} \geq s, D_C^{(1)} \leq s^{\frac{2\alpha_{\kappa}}{\alpha_1}}\bigr) 
    &= \mathbb{P}_C\bigl(D_C^{(\kappa)} \geq s\bigr) -\mathbb{P}_C\bigl(D_C^{(\kappa)} \geq s, D_C^{(1)} > s^{\frac{2\alpha_{\kappa}}{\alpha_1}}\bigr)\\
    &\geq c^{-1}(s^{-(\alpha_{\kappa}+\varepsilon)} - s^{-2\alpha_{\kappa}+\varepsilon}).
\end{align*}
For $s$ big enough we then have
\begin{equation*}
    c^{-1}(s^{-(\alpha_{\kappa}+\varepsilon)} - s^{-2\alpha_{\kappa}+\varepsilon}) \geq \tfrac{1}{2}c^{-1}s^{-(\alpha_{\kappa}+\varepsilon)}, 
\end{equation*}
since for $\varepsilon>0$ small enough we have $-(\alpha_{\kappa}+\varepsilon) > -2\alpha_{\kappa}+\varepsilon$. This gives us the following lower bound for the second probability factor of \eqref{eq:TMM_key_step} 
\begin{equation*}
    \mathbb{P}_C\Biggl(
    \begin{array}{c}
        D_{C}^{_{(\kappa)}} \geq 2^{2(d-1){+\epsilon}}f_{_{n+1}},  \\
        D_{C}^{_{(1)}} \leq 2^{2(d-1)+\epsilon}f_{_{n+1}}^{\frac{2\alpha_{\kappa}}{\alpha_1}}
    \end{array}
    \Biggr) \geq c f_{_{n+1}}^{-(\alpha_{\kappa}+\varepsilon)}.
\end{equation*}
Following now the rest of the calculation as in \cite{Gracar2024a} we get
\begin{equation*}
    \mathbb{P}_x((A_n^x)^c\,|\,A_{n-1}^x) \leq \exp\bigl(  -u c f_n^{\epsilon(\alpha_{\kappa} -\kappa)/2}\bigr).
\end{equation*}
which finally gives us 
\begin{align*}
    \sum\limits_{n=1}^{\infty} \mathbb{P}_x(B_n^{(x)}) &\leq \sum\limits_{n=1}^{\infty}\exp\Bigl(-uc f_n^{\epsilon\frac{\alpha_{\kappa}-\kappa}{2}}\Bigr)\\
    &= \sum\limits_{\ell=0}^{\infty}\exp\Bigl(-u c \bigl\{f_0^{\epsilon(\alpha_{\kappa}-\kappa)/2}\bigr\}^{\bigl(\frac{\min\{d-\kappa,\kappa\}}{\alpha_{\kappa}-\kappa} -\epsilon \bigr)^\ell} \Bigr) =: \delta(f_0).
\end{align*}
This can be made arbitrary small by choosing $f_0$ big enough. Putting all of the bounds together, we obtain
\begin{align*}
    \mathbb{P}_{x,y}(\dist(\mathbf{x},\mathbf{y})\leq 2\Delta) &\leq \sum\limits_{n=1}^{\Delta} \mathbb{P}_{x}(B_n^{(x)}) +\sum\limits_{n=1}^{\Delta}\mathbb{P}_{y}(B_n^{(y)}) +\sum\limits_{n=1}^{2\Delta} \mathbb{P}_{x,y}(G_n^{(x,y)})\\
    &<     2\delta(f_0) \overset{f_0\rightarrow\infty}{\longrightarrow} 0. 
\end{align*}
\qed
\begin{rem}\label{rem:langsam}
    If $\alpha_k\geq\min\{2k,d\}$ for all $k\in\{1,..,d\}$ and therefore $M=\emptyset$, we use a standard almost sure coupling argument as follows. 
    For an arbitrary realisation of the model we extend for all $\mathbf{x}\in\mathcal{X}$ their first diameters $D_x^{_{(1)}}$ using the following deterministic transformation. Let $F$ be the cumulative distribution function of the first diameter in the model, and $G$ the cumulative distribution function of an arbitrary regularly varying random variable with index $2-\rho$ for small $\rho>0$. We replace a given first diameter $\ell$ by $G^{-1}(F(\ell))$ when this value is larger than $\ell$ and keep it as before otherwise. One can quickly check that the resulting model dominates the original one in the sense that all extended convex bodies contain their original versions; all of the bodies also remain convex and their further diameters remain unchanged. Crucially, the tail index of the first diameter in this updated model is $2-\rho$ and therefore the updated model has $M=\{1\}$ and $\alpha_k>k$ for all $k\in\{1,...,d\}$, so the main statement of Theorem \ref{Theorem} applies. Since this model dominates the original model for each realisation, the lower bound for its chemical distance is a lower bound for the original model as well.
    
    As proven, the lower bound is then with high probability equal to $\tfrac{2-\delta}{\log(2/(2-\rho))} \log\log|x-y|$ for all $\delta>0$. The factor in front of the $\log\log$ term is strictly decreasing and converges to infinity as $\rho\downarrow 0$. Consequently, we get that for models with $M=\emptyset$ such that $\alpha_k\geq\min\{2k,d\}$ for all $k\in\{1,\dots,d\}$ the chemical distance is bigger than $c\log\log|x-y|$ with high probability for any $c>0$.
\end{rem}

\section{Proof of the upper bound}\label{Sect:upper}
In this section we first focus on the case that $M\neq\emptyset$ and $\alpha_k>k$ for all $k\in\{1,\dots,d\}$. The proof of the upper bound will be shown for a fixed $k\in\{1,\dots,d-1\}$ which satisfies $\alpha_k \in (k, \min\{2k,d\})$, i.e.\ $k\in M$. Since the result will hold uniformly in $k$, it will in particular also hold for $\kappa$ as defined in Theorem \ref{Theorem}.
We consider the ``smallest'' Boolean model satisfying the condition $\alpha_k \in (k, \min\{2k,d\})$, which has the first $k$ diameters almost surely of the same size and regularly varying with index $-\alpha_k$. The other diameters are set to be deterministic of size $\epsilon>0$. Since we are interested in smallest convex bodies, we furthermore look at the $d$ dimensional convex polytopes as mentioned in Section 2, i.e.\ the convex hulls of $2d$ points which represent the end-points of the diameters and are therefore also the corners of the polytopes. Due to Lemma \ref{Lemma:one}, each such convex polytope with diameters $D^{_{(1)}},\dots,D^{_{(d)}}$ contains a box with side-lengths $2^{-2(d-1)}D^{_{(1)}},\dots,2^{-2(d-1)}D^{_{(d)}}$. As discussed above Lemma \ref{Lemma:two} and used in its proof, we work from here on out with these boxes instead of the larger polytopes that they are contained in. We can do this without affecting our final result, as we explain next.

Recall the definition of $\mathbb{P}=\,^u\mathbb{P}$ and the definition of $\theta(u)$; we work here with a fixed vertex intensity $u>0$ and so write $\mathbb{P}$ and $\theta$ instead of $^u\mathbb{P}$ and $\theta(u)$. Furthermore, we omit the scaling factor $2^{-2(d-1)}$ and assume that the boxes have side-lengths $D^{_{(1)}},\dots,D^{_{(d)}}$. We can do this without loss of generality, since constant factors do not affect the regularly varying distribution of the diameters (and by extension side-lengths of our boxes), and the Poisson point process of intensity $u$ rescaled by a factor $2^{-2(d-1)}$ remains a Poisson point process (with intensity $2^{-2(d-1)}u$). Consequently, using that we are in the robust regime of the model, we can start with intensity $2^{-2(d-1)}u$ instead of $u$ for the original model and proceed with our assumption.
Furthermore, using the same property of regularly varying distributions as above, we will treat the side-lengths of the boxes as if they were the \emph{de facto} diameters in order to keep geometric arguments easier to follow and the notation more concise (if we were being precise the calculations would, up to constant multiplicative factors, remain the same). We will use for the remainder of this section the terms \emph{convex body} and \emph{box} interchangeably, depending on which helps understand the current argument better.

With this we can proceed to the proof of the upper bound. We use a similar argument as in \cite{Hirsch2020} and \cite{Gracar2022a}, which relies on a classical sprinkling argument to connect various parts of the graph with each other.

Recall the definition of a ball $B_r(y)$ for $r>0$, $y\in\mathbb{R}^d$. We focus first on the regions around of $0$ and $x$, namely $B_{\frac{3}{8}|x|}(0)$ (resp.\ $B_{\frac{3}{8}|x|}(x)$) and show that $\mathbf{0}$ (resp.\ $\mathbf{x}$) is connected via a path to a vertex $\mathbf{z}_0$ (resp.\ $\mathbf{z}_x$), with $D^{_{(k)}}_{z_0}> f_0$ (resp.\ $D^{_{(k)}}_{z_x}> f_0$) in this region, where $f_0$ will be chosen later and the entire path is contained inside $B_{\frac{3}{8}|x|}(0)$ (resp.\ $B_{\frac{3}{8}|x|}(x)$). Note that $f_0$ above is the first element of the threshold sequence that was introduced in Section \ref{Sect:lower}.

Recall from Section \ref{Sect:lower} the definition $(\bar A_n^z)_{n\in\mathbb{N}}$ for a given vertex $\mathbf{z}\in\mathcal{X}$. As promised in Remark \ref{rem:A_bar_kappa}, we remind the reader now that from here on out, when referring to the arguments that we used in Section \ref{Sect:lower} they are to be read with $k\in M$ instead of with $\kappa$.  Due to \cite{Gracar2024a} we know that starting with a convex body with $D_z^{(k)}> f_0$, for some $\mathbf{z}\in\mathcal{X}$, the probability that we cannot find a path consisting of convex bodies which fulfil the events $(\bar A_n^z)_{n\in\mathbb{N}}$ 
is smaller than
\begin{equation}\label{eq:Gamma}
    2\left(\exp(-uc f_0^{d-\alpha_k-\epsilon} ) + \sum\limits_{\ell=0}^{\infty}\exp\Bigl(-u c \bigl\{f_0^{\epsilon(\alpha_k-k)/2}\bigr\}^{\bigl(\frac{\min\{d-k,k\}}{\alpha_k-k} -\epsilon \bigr)^\ell} \Bigr)\right)=: \Gamma(f_0).
\end{equation}
As in \cite[Section 2.2]{Gracar2024a}, $f_0$ can be chosen
large enough to guarantee $\Gamma(f_0)<1$. To that end, let now $\rho:=\rho(u,d,k,\alpha_k,\epsilon)>1$ be such to ensure $\Gamma(\rho)<1$. Consider now the event $Z_n$ that no path fulfilling the events $(\bar A_i^z)_{i\in\mathbb{N}}$ with $f_0=\rho^n$ exists. By the above, the probability of this event is $\Gamma(\rho^n)$ and furthermore, by \eqref{eq:Gamma} we also have
\begin{equation*}
    \sum\limits_{n=1}^\infty \mathbb{P}(Z_n)=\sum\limits_{n=1}^\infty \Gamma(\rho^n) <\infty.
\end{equation*}
By Borel-Cantelli, it follows that
\begin{equation*}
    \mathbb{P}(\liminf\limits_{n\rightarrow\infty} Z_n^c)=1.
\end{equation*} 
Consequently, 
there exists an almost surely finite $m\in\mathbb{N}$ such that starting the threshold sequence with $f_0=\rho^{m}$, there exists almost surely a path that satisfies the events $(\bar A_n^z)_{n\in\mathbb{N}}$.
Using this to get an infinite path we claim that the paths starting from $\mathbf{0}$ and from $\mathbf{x}$ are connected to vertices that fulfil the condition $D^{_{(k)}}_{z_0},D^{_{(k)}}_{z_x}\geq \rho^{m}$ and that both paths (which exists almost surely by the above argument) are connected with each other sufficiently early along each of them, by using a single further vertex, all of which occurs with high probability. From here on out, when requiring that $f_0$ is large enough, we will implicitly also assume that the above construction is followed.

\paragraph{Connecting $\mathbf{0}$ and $\mathbf{x}$ to infinite paths in almost surely bounded many steps.}

We need to ensure that the construction of the path from $\mathbf{0}$ to $\mathbf{z}_0$ (resp.\ $\mathbf{x}$ to $\mathbf{z}_x$) is independent of the path between $\mathbf{z}_0$ and $\mathbf{z}_x$, so we use the superposition and thinning properties of the Poisson point process to split it into two independent thinned versions, similarly to how it is done in \cite{Hirsch2020} or \cite{Gracar2022a}. With that in mind, let $b\in(0,1)$ be an arbitrary number; we colour each vertex of $\mathcal{X}$ black with probability $b$ and red with probability $r:=1-b$, independently of everything else. We write $\mathscr{G}^b_u$ (resp.\ $\mathscr{G}^r_u)$ for the graph induced by the black (resp.\ red) vertices in $\mathscr{G}_u:=\mathscr{G}$, where we write $u$ to emphasise the dependence of the model on the density parameter $u$. Moreover let $C_{\infty}^b$ (resp.\ $C_{\infty}^r$) be the unbounded connected component in $\mathscr{G}^b_u$ (resp.\ $\mathscr{G}^r_u$). Since we are working in the robust regime of the model, both of these components exist almost surely.

We define $E_0^b(n,s,v)$ to be the event that there exists a vertex $\mathbf{z}_0$ in the graph $\mathscr{G}^b_u$ with $D^{_{(k)}}_{z_0}>s$ such that $|z_0-0|<v$ and $\mathbf{z}_0$ and $\mathbf{0}$ are connected via at most $n$ further convex bodies. Similarly we define $E_x^b(n,s,v)$ for $\mathbf{x}$. In addition to that we define for $\mathbf{z}_0\in\mathcal{X}$ $E_{0,z_0}^b(n,s,v)$ as the event that $E_0^b(n,s,v)$ occurs and the vertex with location $z_0$ is the vertex that makes it occur. Similarly, we write $E_{x,z_x}^b(n,s,v)$ for $\mathbf{z}_x\in\mathcal{X}$ as the corresponding asked vertex in $E_{x}^b(n,s,v)$. Moreover we also note that $\{\mathbf{0}\leftrightarrow \mathbf{x}\}\cap \{\mathbf{0},\mathbf{x}\in C_{\infty}\}$ converges from below to the event $\{\mathbf{0},\mathbf{x}\in C_{\infty}\}$ as $|x|$ goes to infinity  due to the uniqueness of the unbounded component (see \cite{Chebunin2024}). 

To prove the full claim it suffices to show that for every $s>0$ there exists almost surely a finite random variable $N(s)$ such that 
\begin{equation}\label{eq:upperbound_theta}
    \lim_{b\uparrow 1} \liminf_{s\rightarrow \infty} \liminf_{|x|\rightarrow \infty} \mathbb{P}_{0,x} (\{\mathbf{0},\mathbf{x} \in C_{\infty}^b\} \cap E_0^b(N(s),s,|x|/8) \cap E_x^b(N(s),s,|x|/8) \cap F) \geq \theta^2,
\end{equation}
where $F$ is the event that there exists a path such that $\mathbf{0}$ and $\mathbf{x}$ are connected in fewer than $(2+\delta)\log\log|x|/\bigl(\log(\min\{d-k,k\})-\log(\alpha_k-k)\bigr))$ steps for small $\delta>0$.  
We show first
\begin{equation*}
    \liminf_{s\rightarrow \infty} \liminf_{|x|\rightarrow \infty} \mathbb{P}_{0,x} (\{\mathbf{0},\mathbf{x} \in C_{\infty}^b\} \cap E_0^b(N(s),s,|x|/8) \cap E_x^b(N(s),s,|x|/8) \cap F) \geq \theta_b^2,
\end{equation*}
from which \eqref{eq:upperbound_theta} follows by Proposition \ref{Prop:one}. It holds that 
\begin{align}
    \mathbb{P}_{0,x} (&\{\mathbf{0},\mathbf{x} \in C_{\infty}^b\} \cap E_0^b(N(s),s,|x|/8) \cap E_x^b(N(s),s,|x|/8) \cap F)\notag\\
    &= \mathbb{E}_{0,x} \Bigl[\mathbb{P}_{0,x,z_0,z_x} (\{\mathbf{0},\mathbf{x} \in C_{\infty}^b\} \cap E_{0,z_0}^b(N(s),s,|x|/8) \cap E_{x,z_x}^b(N(s),s,|x|/8) \cap F) \Bigr]\notag \\
    &\geq \mathbb{P}_{0,x} ( \{\mathbf{0},\mathbf{x} \in C_{\infty}^b\} )\notag\\
    &\hspace{20pt} - \mathbb{P}_{0,x} ( \{\mathbf{0}\in C_{\infty}^b\} \setminus E_0^b(N(s),s,|x|/8) ) \notag\\
    &\hspace{20pt} - \mathbb{P}_{0,x} ( \{\mathbf{x}\in C_{\infty}^b\} \setminus E_x^b(N(s),s,|x|/8) ) \label{eq:upperbound_key_calculation}\\
    &\hspace{20pt}-  \mathbb{E}_{0,x} \Bigl[\mathbb{P}_{0,x,z_x,z_0}\bigl((E_{0,z_0}^b(N(s),s,|x|/8)\cap E_{x,z_x}^b(N(s),s,|x|/8)) \setminus F\bigr)\Bigr].\notag
\end{align}
We now show that the last three summands converge to zero as $s$ and $|x|$ are made large.
Due to the translation invariance of the Poisson process, the second and third summands are equal.
Considering the limit $s\rightarrow \infty$ and using the same arguments as for the proof in \cite[Lemma~3.2]{Gracar2022a} gives that the second and third term converge to $0$, since the probability of arbitrary big convex bodies belonging to the unbounded connected component is positive.
It remains to consider the last summand of \eqref{eq:upperbound_key_calculation}. For this we need two things. First, we will use the properties of the infinite path that satisfies $(\bar A_n^{z_0})_{n\in\mathbb{N}}$ (resp.\ $(\bar A_n^{z_x})_{n\in\mathbb{N}}$), so that we have two growing sequences of convex bodies. Second, we will find a single convex body to ensure that the two infinite paths are connected to each other ``early on''.

\paragraph{Reaching sufficiently powerful vertices in fewer than $\tfrac{(1+\delta)}{\log\bigl( \frac{\min\{d-k,k\}}{\alpha_k - k}\bigr)}\log\log|x-y|$ many steps, with high probability.}
As argued before, we know that an infinite path exists when starting with a convex body which is large enough, namely if its threshold sequence begins with $f_0>0$ big enough. We focus now on the infinite path started from $\mathbf{z}_0$ using that $(\bar A_n^{z_0})_{n\in\mathbb{N}}$ holds and denote by $n_1$ the first index $i$ in the path, for which the corresponding threshold $f_i\geq|x|/8$. For the path starting in $\mathbf{z}_0$ we denote this convex body with its location by $\mathbf{y}_0$ (resp. $\mathbf{y}_x$ for the path starting in $\mathbf{z}_x$). To avoid confusion, note that $\mathbf{y}_0$ is not the $n_1$th vertex of the path started in $\mathbf{0}$, but is instead only the $n_1$th vertex \emph{after} $\mathbf{z}_0$, which is itself almost surely finitely many steps along the path. The same observation of course holds also for $\mathbf{y}_x$. Let now, as before, $f_0$ be large enough. This leads to
\begin{align*}
     &f_{n_1} \geq |x|/8\\
    &\Leftrightarrow f_0^{\bigl(\frac{\min\{d-k,k\}}{\alpha_k-k}-\epsilon\bigr)^{n_1}} \geq|x|/8 \\
     &\Leftrightarrow n_1\log \bigl(\frac{\min\{d-k,k\}}{\alpha_k-k}-\epsilon\bigr) + \log\log f_{{0} } \geq\log\log (|x|/8).
\end{align*}
Recall that we are interested in an upper bound for the chemical distance. We choose therefore
\begin{equation*}
    n_1 \geq \Biggl\lceil \frac{\log\left(\log |x| - \log(8)\right)-\log\log f_{_{0} }}{\log \bigl(\frac{\min\{d-k,k\}}{\alpha_k-k}-\epsilon\bigr)}  \Biggl\rceil.
\end{equation*}
We can assume without loss of generality that $|x|$ is large enough so that the locations of the vertex $\mathbf{y}_0$ as well as all preceding vertices is in $B_{\frac{3}{8}|x|}(0)$ and by translation invariance the same is true for $\mathbf{y}_x$ and the ball $B_{\frac{3}{8}|x|}(x)$. We can assume without loss of generality that $f_{n_1}= |x|/8$ as we have that $\lceil |x|\rceil/|x|$ converges to $1$ as $|x|\rightarrow \infty$. In all other cases we get some (bounded) multiplicative factors in the calculations which do not affect the results.

Before proceeding to the final missing step, let us summarise our work so far. Once \eqref{eq:upperbound_theta} is established, we have that on the event that $\{\mathbf{0},\mathbf{x}\in C_\infty\}$, one can in at most almost surely finitely many (in fact and crucially, almost surely bounded many) steps connect $\mathbf{0}$ to some vertex $\mathbf{z}_0$ with high probability as $|x|\to\infty$. The convex body of this vertex is then sufficiently large that it is almost surely the first vertex of an infinite path satisfying the events $(\bar A_n^{z_0})_{n\in \mathbb{N}}$, whose $n_1$th vertex $\mathbf{y}_0$ has the $k$th side-length larger than $|x|/8$ and the path from $\mathbf{z}_0$ to $\mathbf{y}_0$ takes fewer than $(1 +\delta)\log\log |x|/\log \bigl(\frac{\min\{d-k,k\}}{\alpha_k-k}\bigr)$ many steps for small $\delta>0$; the same is true also for the vertex $\mathbf{x}$. Once we establish that $\mathbf{y}_0$ and $\mathbf{y}_x$ are connected to the same vertex with an even larger convex body with high probability, \eqref{eq:upperbound_theta} will be proven and the proof of the upper bound complete.
The above construction is outlined in Figure \ref{fig:upper_bound_sketch}, showing how the path connecting $\mathbf{0}$ with $\mathbf{x}$ arises.

\begin{figure}[!ht] \begin{annotate}{
              \includegraphics[width=1\linewidth]{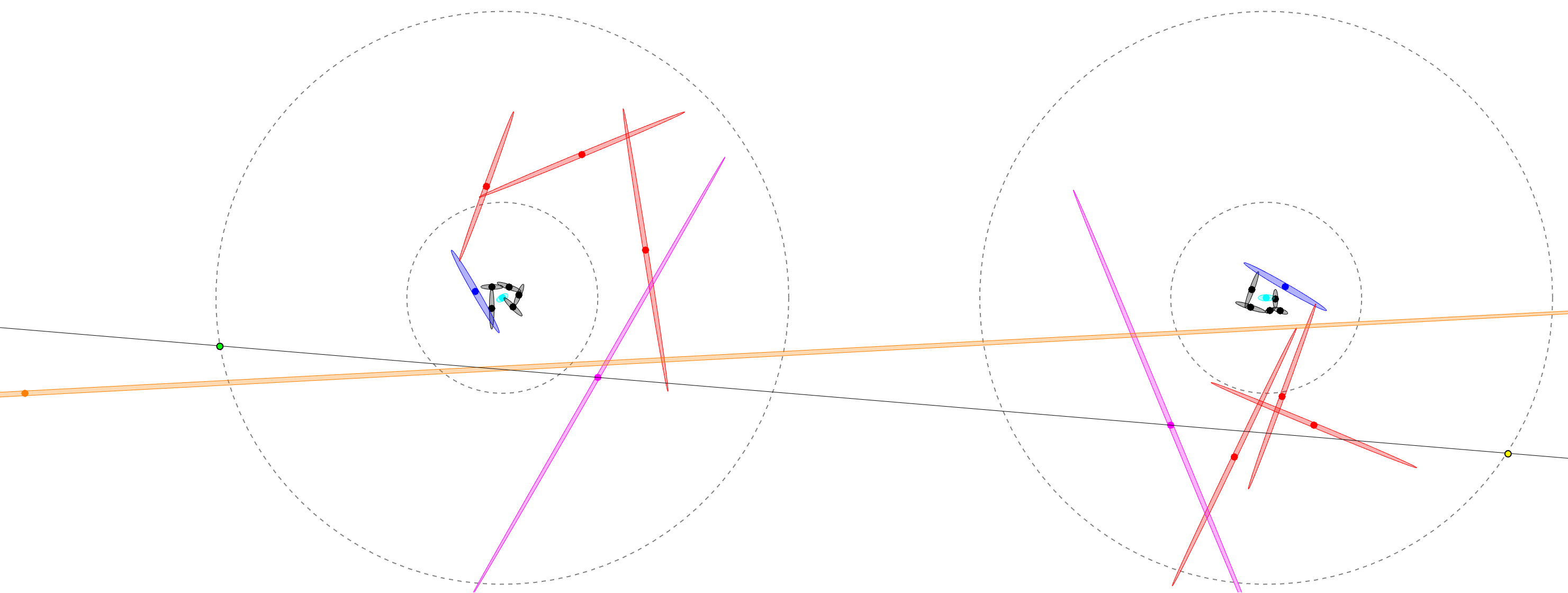}}{1}
              \draw (-2.62,0) -- (-2.4,0.3);
              \node at (-2.3,0.4) {$\mathbf{0}$};
              \draw (4.5,0) -- (4.7,0.3);
              \node at (4.75,0.4) {$\mathbf{x}$};
              \draw (-2.85,-0.05) -- (-3.1,-0.3);
              \node at (-3.2,-0.4) {$\mathbf{z}_0$};
              \draw (4.85,0) -- (5.2,0);
              \node at (5.4,0) {$\mathbf{z}_x$};
              \draw (-1.88,-1) -- (-1.65,-1.3);
              \node at (-1.5,-1.5) {$\mathbf{y}_0$};
              \draw (3.722,-1.45) -- (3.55,-1.75);
              \node at (3.4,-1.9) {$\mathbf{y}_x$};
              \draw (-5.29,-0.438) -- (-5.45,-0.25);
              \node at (-5.65,-0.15) {$v_0$};
              \draw (6.75,-1.45) -- (6.6,-1.3);
              \node at (6.45,-1.15) {$v_x$};
              \draw (-6.95,-0.9) -- (-6.75,-1.2);
              \node at (-6.6,-1.35) {$\mathbf{y}$};
              \draw[->] (-3.9,1) -- (-3.4,0.5);
              \node at (-4,1.25) {$\partial B_{|x|/8}(0)$};
              \draw[->] (-5.3,-1.9) -- (-4.68,-1.75);
              \node at (-6.2,-2.0) {$\partial B_{3|x|/8}(0)$};
              \draw[->] (5.75,1) -- (5.25,0.5);
              \node at (6,1.25) {$\partial B_{|x|/8}(x)$};
              \draw[->] (1.95,-1.9) -- (2.45,-1.75);
              \node at (1.03,-2.0) {$\partial B_{3|x|/8}(x)$};
              \draw[line width=2pt] (0.025,0.3) -- (0.025,-0.3);
              \draw[line width=2pt] (1.85,0.3) -- (1.85,-0.3);
              \draw[->] (0.6,0) -- (0.035,0);
              \draw[->] (1.205,0) -- (1.84,0);
              \node at (0.9025,0) {$\tfrac{|x|}{2}$};
            \end{annotate}
              \caption{In cyan blue $\mathbf{0}$ and $\mathbf{x}$, in black the path connecting $\mathbf{0}$ with $\mathbf{z}_0$ (resp.\ $\mathbf{x}$ with $\mathbf{z}_x$), in blue $\mathbf{z}_0$ and $\mathbf{z}_x$, the path connecting $\mathbf{z}_0$ with $\mathbf{y}_0$ in red (resp.\ $\mathbf{z}_x$ with $\mathbf{y}_x$), $\mathbf{y}_0$ and $\mathbf{y}_x$ in pink and in orange the vertex $\mathbf{y}$ connecting $\mathbf{y}_0$ and $\mathbf{y}_x$. The dashed gray circles are the boundary of $B_{|x|/8}(0)$, $B_{3|x|/8}(0)$, $B_{|x|/8}(x)$ and $B_{3|x|/8}(x)$. $v_0$ and $v_x$ as defined in \eqref{eq:v0} and \eqref{eq:vx}.}
              \label{fig:upper_bound_sketch}
\end{figure}
\paragraph{Connecting the two sufficiently powerful vertices through a single connecting vertex with high probability.} 

We now proceed to argue this last step. For that we define $F_{n_1}^{(0,z_0)}$ (resp.\ $F_{n_1}^{(x,z_x)}$) as the event that there exists a path of length $n_1$ in the graph $\mathscr{G}^r_u$ such that starting with $\mathbf{z}_0$ (resp.\ $\mathbf{z}_x$) the path satisfies $\bar A_{n_1}^{z_0}$ (resp.\ $\bar A_{n_1}^{z_x}$). In addition to that we define $F_{n_1}^{(0,z_0,y_0)}$ as the event that in $F_{n_1}^{(0,z_0)}$, $\mathbf{y}_0$ is the $n_1$th vertex in the path that implies $\bar A_{n_1}^{z_0}$. Analogously we define $F_{n_1}^{(x,z_x,y_x)}$ for $\mathbf{y}_x$. Moreover let $F^{(0,x,y_0,y_x)}$ be the event that there exists a vertex outside of $B_{|x|^{\frac{\min\{d-k,k\}}{\alpha_k-k}-\epsilon}}(0)$ with a corresponding box that intersects $\mathbf{y}_0$ and $\mathbf{y}_x$. 
Note that under the condition that $0,x,z_0,z_y,y_0,y_x\in\mathscr{P}$ we have that $F_{n_1}^{(0,z_0,y_0)}\cap F_{n_1}^{(x,z_x,y_x)}\cap F^{(0,x,y_0,y_x)}\subseteq F$, conditioned on $E_0^b(D(s),s,|x|/8)\cap E_x^b(D(s),s,|x|/8)$; i.e.\ the start of the path in $\mathbf{0}$ and $\mathbf{x}$ being successful.  
From here on we write $E_{0,z_0}^b$ for $E_{0,z_0}^b(D(s),s,|x|/8)$ (resp.\ $E_{x,z_x}^b$ for $E_{x,z_x}^b(D(s),s,|x|/8)$) to keep the notation concise. We will now focus on the probability in the brackets of the last summand in \eqref{eq:upperbound_key_calculation}. We can rewrite this term to
\begin{equation}\label{eq:last_summand_rewritten}
    \mathbb{P}_{0,x,z_0,z_x}\bigl((E_{0,z_0}^b\cap E_{x,z_x}^b )\setminus F\bigr) = \mathbb{P}_{0,x,z_0,z_x}\bigl( F^c \,|\,E_{0,z_0}^b\cap E_{x,z_x}^b\bigr)\mathbb{P}_{0,x,z_0,z_x}\bigl(E_{0,z_0}^b\cap E_{x,z_x}^b\bigr)
\end{equation}
in order to work with the conditional probability given $E_{0,z_0}^b\cap E_{x,z_x}^b$, and use the events defined earlier to bound the last expression from above by using
\begin{align}
    &\mathbb{E}_{0,x,z_0,z_x}\Bigl[\mathbb{P}_{0,x,z_0,z_x,y_0,y_x}\bigl( (F_{n_1}^{(0,z_0,y_0)}\cap F_{n_1}^{(x,z_x,y_x)}\cap F^{(0,x,y_0,y_x)})^c \,|\,E_{0,z_0}^b\cap E_{x,z_x}^b\bigr)\notag\\
    &\hspace{250pt}\times\mathbb{P}_{0,x,z_0,z_x,y_0,y_x}\bigl(E_{0,z_0}^b\cap E_{x,z_x}^b\bigr)  \Bigr]\notag \\ 
    &\leq \mathbb{P}_{0,x,z_0,z_x}\bigl( (F_{n_1}^{(0,z_0)})^c\cup (F_{n_1}^{(x,z_x)})^c \,|\,E_{0,z_0}^b\cap E_{x,z_x}^b\bigr) \label{eq:upperbound_for_last_summand} \\
    &\hspace{15pt}+\mathbb{E}_{0,x,z_0,z_x}\Bigl[\mathbb{P}_{0,x,z_0,z_x,y_0,y_x}\bigl( (F_{n_1}^{(0,z_0,y_0)}\cap F_{n_1}^{(x,z_x,y_x)})\cap (F^{(0,x,y_0,y_x)})^c \,|\,E_{0,z_0}^b\cap E_{x,z_x}^b\bigr)\notag\\
    &\hspace{250pt}\times\mathbb{P}_{0,x,z_0,z_x,y_0,y_x}\bigl(E_{0,z_0}^b\cap E_{x,z_x}^b\bigr)\bigr].\notag
\end{align}

The first summand is by our choice of $f_0$ almost surely equal to $0$.
In order to show that the last summand in \eqref{eq:upperbound_key_calculation} converges to zero it suffices to show that the second summand in \eqref{eq:upperbound_for_last_summand}
or equivalently
\begin{equation}\label{eq:upperbound_probability_no_connector}
    \mathbb{E}_{0,x,z_0,z_x}\Bigl[\mathbb{P}_{0,x,z_0,z_x,y_0,y_x}\bigl(F_{n_1}^{(0,z_0,y_0)}\cap F_{n_1}^{(x,z_x,y_x)}\cap (F^{(0,x,y_0,y_x)})^c \cap E_{0,z_0}^b\cap E_{x,z_x}^b\bigr) \bigr]
\end{equation}
has a uniform bound over all $z_0$ and $z_x$ which converges to zero as $|x|\rightarrow \infty$; this implies that integrating \eqref{eq:last_summand_rewritten} over $z_0$ and $z_x$ leads to a term that converges to zero as $|x|\rightarrow\infty $.
We focus now on $(F^{(0,x,y_0,y_x)})^c $ and introduce some new terminology to help us with that. Recall that  $p_v^{(i)}$ is the orientation of the $i$th side-length of the box $C_v$ for $\mathbf{v}\in\mathcal{X}$ and $i\in\{1,\dots,d\}$. We say a vertex $\mathbf{v}\in\mathcal{X}$ is \emph{good with respect to} the vertex $\mathbf{w}\in\mathcal{X}$ if it satisfies the following two conditions.
\begin{enumerate}
    \item[(G1)] The vertex $\mathbf{v}$ has \emph{good orientation} relative to the orientation of $\mathbf{w}$ in the sense that for all $i\in\{1,\dots,d\}$ the orientation satisfies
    \begin{equation*}
    p_{v}^{(i)}\in\{\phi \in\mathbb{S}^{d-1}:\, \measuredangle ( \phi,p_{w}^{(i)})\leq \tfrac{\epsilon}{\log|x|}\}.
\end{equation*}

    \item[(G2)] The location $v$ of $\mathbf{v}$ has \emph{good position} relative to $\mathbf{w}$, i.e.\ it satisfies
\begin{equation*}
        v \in  \left\{y\in\mathbb{R}^d: |w-y| \geq 2 f_{n_1}, \measuredangle\big(w-y, \pm p_{v}^{_{(j)}}\big)> \varphi \mbox{ for all }  1\leq j\leq k \right\}=:A(w,n_1),
    \end{equation*}
    where $\varphi=2^{-k}\pi$.

\end{enumerate}
It is useful at this stage to note that by construction $|y_0-y_x|\geq 2f_{n_1}$, which is one of the two conditions in (G2). We also observe (and prove later on) that
when $\mathbf{v}$ is good relative $\mathbf{w}$ we have that the intersection of the orthogonal projections of $C_v$ and $C_w$ onto a hyperplane perpendicular to $v-w$, has $(d-1)$ dimensional Lebesgue measure of order $f_{n_1}^k$. In addition to that (G2) gives us roughly that $v$ does not lie ``too close'' to the affine subspace through $w$ spanned by the $k$ large side-lengths of $C_w$, where $\varphi$ is chosen as in $(\bar A_n^x)_{n\in\mathbb{N}}$. (G2) also ensures that $v$ has distance at least $2f_{n_1}$ to $w$. Note that this is similar to the definition of \eqref{eq:Oi}. 
Observe also that the restriction in (G1) becomes tighter for large $|x|$, and so for sufficiently large $|x|$ we get that $\mathbf{v}$ being good with respect to $\mathbf{w}$ implies almost surely also the reverse relationship, i.e.\ $\mathbf{w}$ is good with respect to $\mathbf{v}$. 

Using this (in particular, that $|x|$ is large) we can rewrite $(F^{(0,x,y_0,y_x)})^c$ as 
\begin{align*}
    (F^{(0,x,y_0,y_x)})^c =& \hspace{10pt}\Bigl(\{\not \exists y\in\mathscr{P}\setminus B_{|x|^{\frac{\min\{d-k,k\}}{\alpha_k-k}-\epsilon}}(0): C_y\cap C_{{y}_x}\neq \emptyset \text{ and }C_y\cap C_{{y}_0}\neq \emptyset\}\\
    &\hspace{10pt}\cap \{\mathbf{y}_0 \text{ and }\mathbf{y}_x \text{ are good with respect to each other.}\}\Bigr)\\
    &\hspace{5pt}\cup \Bigl(\{\not \exists y\in\mathscr{P}\setminus B_{|x|^{\frac{\min\{d-k,k\}}{\alpha_k-k}-\epsilon}}(0): C_y\cap C_{{y}_x}\neq \emptyset \text{ and }C_y\cap C_{{y}_0}\neq \emptyset\}\\
    &\hspace{10pt}\cap \{\mathbf{y}_0 \text{ and }\mathbf{y}_x \text{ are not good with respect to each other.}\}\Bigr)\\
    &\subseteq \Bigl(\{\not \exists y\in\mathscr{P}\setminus B_{|x|^{\frac{\min\{d-k,k\}}{\alpha_k-k}-\epsilon}}(0): C_y\cap C_{{y}_x}\neq \emptyset \text{ and }C_y\cap C_{{y}_0}\neq \emptyset\}\\
    &\hspace{6pt}\cap \{\mathbf{y}_0 \text{ and }\mathbf{y}_x \text{ are good with respect to each other.}\}\Bigr)\\
    & \hspace{5pt}\cup  \{\mathbf{y}_0 \text{ and }\mathbf{y}_x \text{ are not good with respect to each other.}\}\\
    &\hspace{130pt}=: (E_1^{(y_0,y_x)} \cap E_2^{(y_0,y_x)} ) \cup E_3^{(y_0,y_x)} .
\end{align*}
With this we can bound \eqref{eq:upperbound_probability_no_connector} from above by 
\begin{align*} 
    \mathbb{P}_{0,x,z_0,z_x,y_0,y_x}&\bigl(E_1^{(y_0,y_x)} \cap E_2^{(y_0,y_x)}  \cap F_{n_1}^{(0,z_0,y_0)}\cap F_{n_1}^{(x,z_x,y_x)}\cap E_{0,z_0}^b\cap E_{x,z_x}^b\bigr) \\
    &\quad\quad+\mathbb{P}_{0,x,z_0,z_x,y_0,y_x}\bigl( E_3^{(y_0,y_x)}  \cap  F_{n_1}^{(0,z_0,y_0)}\cap F_{n_1}^{(xz_x,y_x)}\cap E_{0,z_0}^b\cap E_{x,z_x}^b\bigr)\\
    & \leq \mathbb{P}_{0,x,z_0,z_x,y_0,y_x}\bigl(E_1^{(y_0,y_x)} \,|\, E_2^{(y_0,y_x)}  \cap F_{n_1}^{(0,z_0,y_0)}\cap F_{n_1}^{(xz_x,y_x)}\cap E_{0,z_0}^b\cap E_{x,z_x}^b\bigr) \\&\quad\quad+\mathbb{P}_{0,x,z_0,z_x,y_0,y_x}\bigl( E_3^{(y_0,y_x)}  \cap F_{n_1}^{(0,z_0,y_0)}\cap F_{n_1}^{(x,z_x,y_x)}\cap E_{0,z_0}^b\cap E_{x,z_x}^b\bigr).
\end{align*}
We will show that this expression has a uniform upper bound over all $z_0,z_x,y_0,y_x$ and that this bound converges to zero as $|x|\rightarrow \infty$. For that we define similarly to \cite{Gracar2024a}  the $\sigma$-algebra $\mathfrak F_n^{x}$, generated by the restriction of the (simplified box) Poisson-Boolean base model to the points with locations inside $B_{3f_n}(0) \cup B_{3f_n}(x) $. Using this we have that 
$E_2^{(y_0,y_x)}$, $ F_{n_1}^{(0,z_0,y_0)}$, $ F_{n_1}^{(x,z_x,y_x)}$, $ E_{0,z_0}^b$, $ E_{x,z_x}^b$ are $\mathfrak F_{n_1}^x$ measurable.
\begin{lemma}\label{Lemma:three}
    For the events defined as above, we have 
    \begin{equation*}
        \mathbb{P}_{0,x,z_0,z_x,y_0,y_x}\!\bigl(E_1^{(y_0,y_x)} | E_2^{(y_0,y_x)}  \cap F_{n_1}^{(0,z_0,y_0)}\cap F_{n_1}^{(x,z_x,y_x)}\cap E_{0,z_0}^b\cap E_{x,z_x}^b\bigr)\!\!\leq \!\exp\bigl(  -u c |x|^{\epsilon(\alpha_k -k)/2}\bigr).
    \end{equation*}
\end{lemma}
\begin{proof}
    We define for $v,w,z\in\mathbb{R}^d$ and $A\subseteq\mathbb{R}^d$ the set $P_{v,w}^{z}(A)$ as the orthogonal projection of $A$ onto the hyperplane through $z$ perpendicular to $w-v$. In addition to that define  
    \begin{equation}\label{eq:v0}
        v_0:=\frac{y_0-y_x}{|y_0-y_x|}\gamma_0(y_0,y_x)+ y_0,
    \end{equation}
    and 
    \begin{equation}\label{eq:vx}
        v_{x}:=\frac{y_x-y_0}{|y_x-y_0|}\gamma_x(y_0,y_x)+y_x
    \end{equation}
    where $\gamma_0(y_0,y_x),\gamma_x(y_0,y_x)>0$ are chosen such that $v_0\in \partial B_{3|x|/8}(0)$ and $v_x\in\partial B_{3|x|/8}(x)$. $v_0$ and $v_x$ can be found in Figure~\ref{fig:upper_bound_sketch}. 
    
    With this we define 
    \begin{equation*}
        T_{v_x}:=P_{y_0,y_x}^{v_x}(C_{y_0})\cap P_{y_0,y_x}^{v_x}(C_{y_x}),
    \end{equation*}
    that is, the intersection of the orthogonal projections of $C_{y_0}$ and $C_{y_x}$ onto the hyperplane through $v_x$ perpendicular to $y_x-y_0$. An example for $T_{v_x}$ in $\mathbb{R}^3$ is pictured in Figure~\ref{fig:connector_of_y0yx}. 

    Since we are conditioning on the event that $\mathbf{y}_0$ and $\mathbf{y}_x$ are good with respect to each other, we can control the size of the area of $T_{v_x}$.
    Recall that we assume $|x|$ is large enough for the \emph{good with respect to} property to be symmetric, so that $T_{v_x}$ is a $(d-1)$ dimensional polytope with diameters $D_T^{(1)},\dots,D_T^{(d-1)}$ which satisfy $D_T^{(i)}\geq \tfrac{1}{2}\varphi\min\bigl\{D_{y_0}^{(i)},D_{y_x}^{(i)}\bigr\}$ for all $i\in\{1,\dots,d-1\}$. 
    Consequently, the $(d-1)$ dimensional Lebesgue mass of $T_{v_x}$ is at most
    \begin{equation*}
        c \prod\limits_{i=1}^{d-1} \min\bigl\{D_{y_0}^{(i)},D_{y_x}^{(i)}\bigr\},
    \end{equation*}
    where $c$ is a constant that depends only on $d$ and $k$. 
    Furthermore, $D_T^{(1)},\dots ,D_T^{(k)}\geq \tfrac{1}{2}\varphi f_{n_1}$ and $D_T^{(k+1)},\dots ,D_T^{(d-1)}\geq \tfrac{1}{2}\varphi\epsilon$. 
    
    Using now Lemma~\ref{Lemma:one} we know that there exists a $(d-1)$ dimensional box contained in $T_{v_x}$ which is congruent to 
    \begin{equation*}
        c\cdot\bigl([0,2^{-2(d-2)}f_{n_1}]^{k}\times [0,2^{-2(d-2)}\epsilon]^{d-1-k}\times\{0\}\bigr).
    \end{equation*}
    We denote by $B_{T_{v_x}}$ the box which satisfies this. This box is in general not unique, so we choose without loss of generality $B_{T_{v_x}}$ using an arbitrary lexicographic ordering of the boxes with respect to their centers, orientations and side-lengths. Note that the set of all possible boxes is closed and as such, a minimal (or maximal) element of can be uniquely chosen.
   In addition to that we define $\rot_{y_0,y_x}$ as the rotation that maps the canonical basis vectors $e_1,\dots,e_d$ of $\mathbb{R}^d$ as follows,
    \begin{align*}
        \rot_{y_0,y_x}(e_1)&= \tfrac{y_x - y_0} {|y_x-y_0|}\\
        \rot_{y_0,y_x}(e_i)&= v_i,\hspace{5pt}i\in\{2,\dots,d\},
    \end{align*}
    where $v_i$ is defined as the orientation of the $(i-1)$st diameter of $B_{T_{v_x}}$, for $i\in\{2,\dots,d\}$. Note that the choice of $(v_i)_{i\in\{2,\dots,d\}}$ is unique up to permutations and mirrorings of the side-lengths of $B_{T_{v_x}}$ (and we pick the first of these according to some arbitrary ordering). 
    We say a vertex $\mathbf{y}\in\mathcal{X}$ with corresponding box that intersects $\mathbf{y}_0$ and $\mathbf{y}_x$ is \emph{good} if it has the following properties.
    \begin{enumerate}
        \item[(C1)] $C_y$ intersects $P_{y_0,y_x}^{v_0}(B_{T_{v_x}})$ and $\tfrac{1}{3}(B_{T_{v_x}}-v_x)+v_x$.
        \item[(C2)] The $k$th side-length of $C_y$ is at least $2(|v_x-v_0|+|y-v_0|)$.
        \item[(C3)] For the location of $\mathbf{y}$ we have
        \begin{align*}
            y\in I_{y_0,y_x}:=\rot_{y_0,y_x}\biggl( \Bigl\{&y=(y^{(1)},\dots,y^{(d)})\in\mathbb{R}^d: y^{(1)}\leq -|x|^{\frac{\min\{d-k,k\}}{\alpha_k-k}-\epsilon},\\
            & y^{(i)}\in\bigl[-c_1(|y^{(1)}|+|v_0-v_x|),c_1(|y^{(1)}|+|v_0-v_x|)\bigr], \\
            &  y^{(j)}\in\bigl[-c_2(|y^{(1)}|+|v_0-v_x|),c_2(|y^{(1)}|+|v_0-v_x|)\bigr],\\
            &\quad \quad \quad \quad \quad \quad  i\in\{2,\dots,k+1\}, j\in\{k+2,\dots,d\}\Bigr\}\biggr) \!+\!v_0,
        \end{align*}
    \end{enumerate}
    where 
    \begin{equation*}
        c_1:=  \frac{\tfrac{1}{3}\cdot2^{-2(d-2)}f_{n_1}}{|v_0-v_x|}
    \end{equation*}
    and
    \begin{equation*}
        c_2: = \frac{\tfrac{1}{3}\cdot2^{-2(d-2)}\epsilon}{|v_0-v_x|}.
    \end{equation*}

    (C1) is chosen such that it implies that $C_y$ intersects with $C_{y_0}$ and $C_{y_x}$. Next, (C2) guarantees that the first $k$ side-lengths of $C_y$ are large enough to result in an intersection with $\tfrac{1}{3}(B_{T_{v_x}}-v_x)+v_x$ when the orientation of the $C_y$ is suitable. (C3) ensures that the argument we will use works for all $k\in\{1,\dots,d-1\}$ and that we do not have to treat the case $k=1$ separately. In particular it, together with (C2), ensures that if $C_y$ intersects $\tfrac{1}{3}(B_{T_{v_x}}-v_x)+v_x$, it must also intersect $P_{y_0,y_x}^{v_0}(B_{T_{v_x}})$. 
    Figure~\ref{fig:connector_of_y0yx} depicts the sets (C1) and (C3) for the $3$ dimensional case and Figure~\ref{fig:c1c2_for_connector} illustrates the intuition behind the constants $c_1$ and $c_2$.  
    \begin{figure}[!t]
        \begin{subfigure}[h]{.495\linewidth}
            \includegraphics[width=\linewidth]{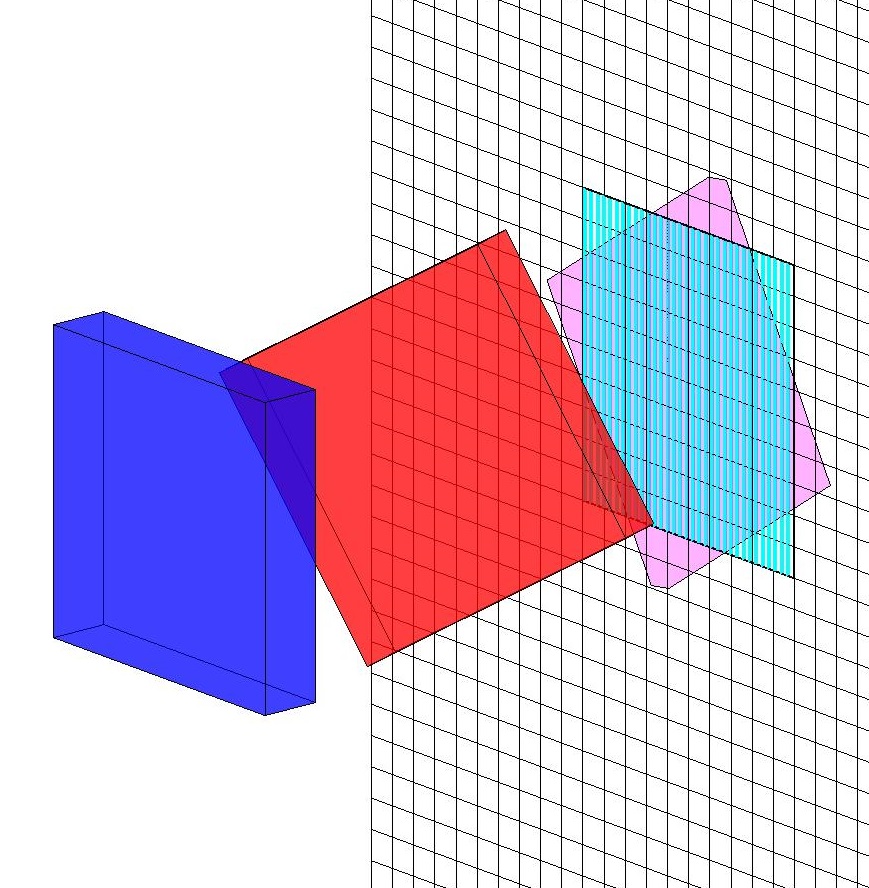}
            \subcaption{The hyperplane perpendicular to $y_0-y_x$ which includes $v_x$ depicted as a grid, $C_{y_0}$ in blue and its orthogonal projection in turquoise, $C_{y_x}$ in red and its orthogonal projection in pink.}
        \end{subfigure}
        \begin{subfigure}[h]{.495\linewidth}
            \includegraphics[width=\linewidth]{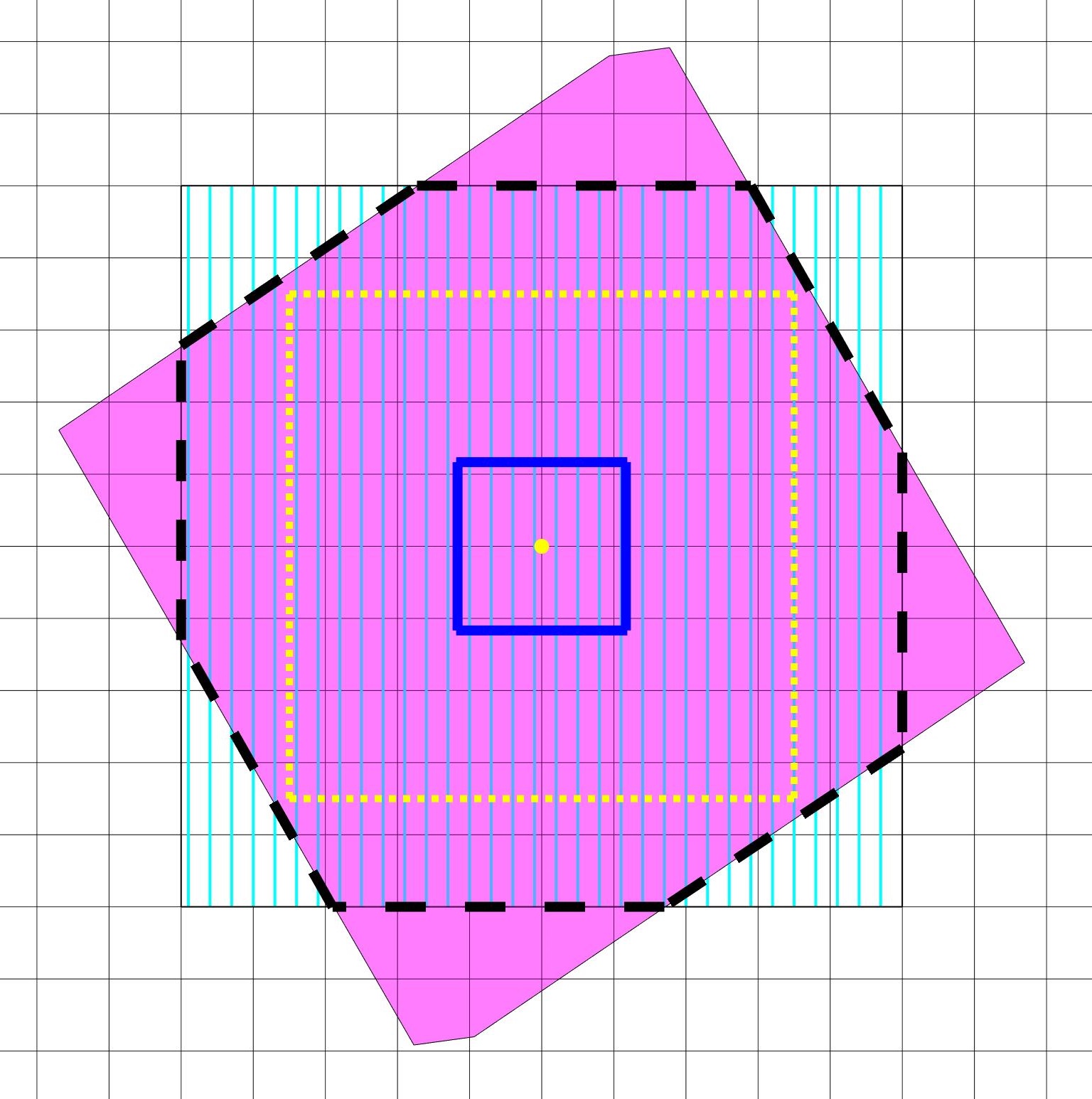}
            \subcaption{Dashed black line the boundary of \smash{$T_{v_x}$}, dotted yellow line the boundary of \smash{$B_{T_{v_x}}$}, in blue the boundary of  \smash{$\tfrac{1}{3}(B_{T_{v_x}}-v_x)+v_x$}, yellow point is \smash{$v_0$}.}
        \end{subfigure}
        \begin{subfigure}[h]{1\linewidth}
            \includegraphics[width=\linewidth]{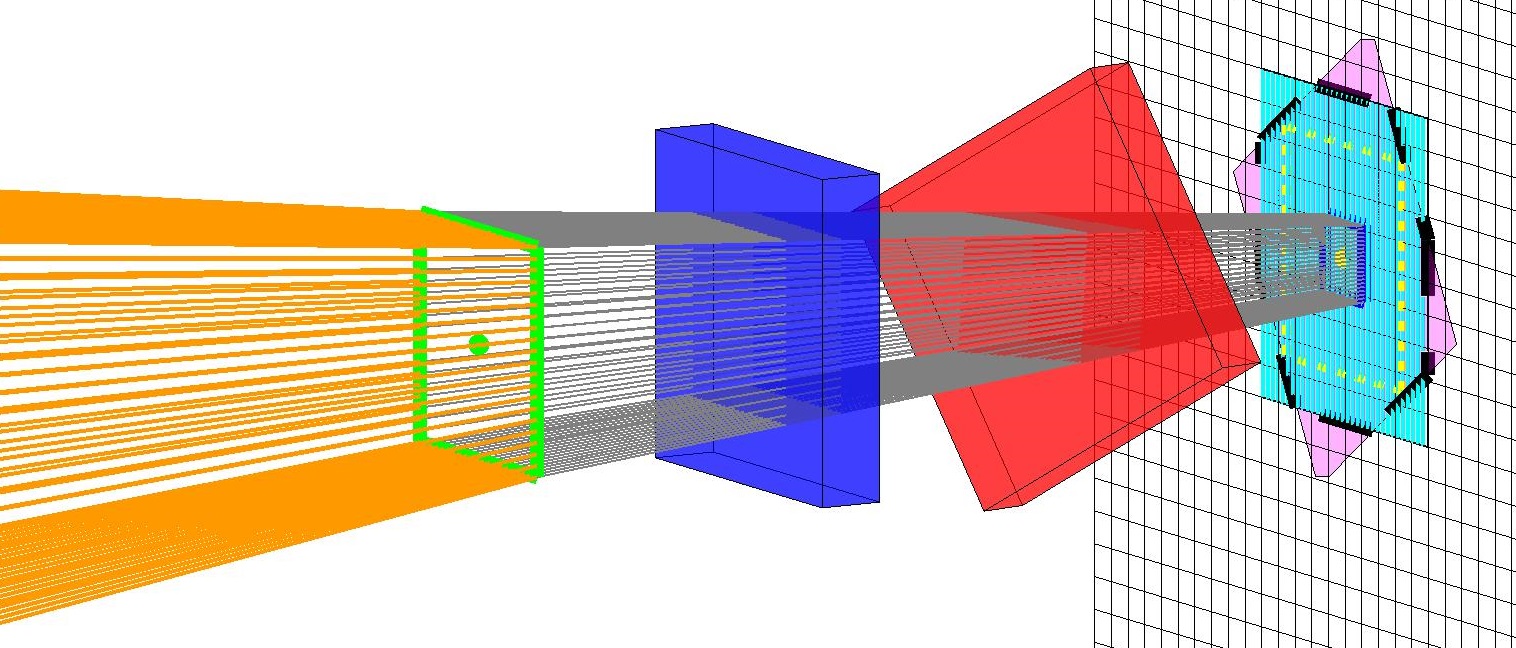}
            \subcaption{Green lines are the boundary of $P_{y_0,y_x}^{v_0}(B_{T_{v_x}})$. The gray lines as the construction to get the allowed positions for the location of $\mathbf{y}$, i.e.\ $I_{y_0,y_x}$, whose boundary is in orange.}
        \end{subfigure}
    \caption{Example in $\mathbb{R}^3$. Visualisation of the possible defined sets in (C1) and (C3) for the connector $\mathbf{y}$ for $\mathbf{y}_0$ and $\mathbf{y}_x$.}\label{fig:connector_of_y0yx}
    \end{figure}

        \begin{figure}[!hb]
        \begin{subfigure}{1\textwidth}\begin{annotate}{
              \includegraphics[width=1\linewidth]{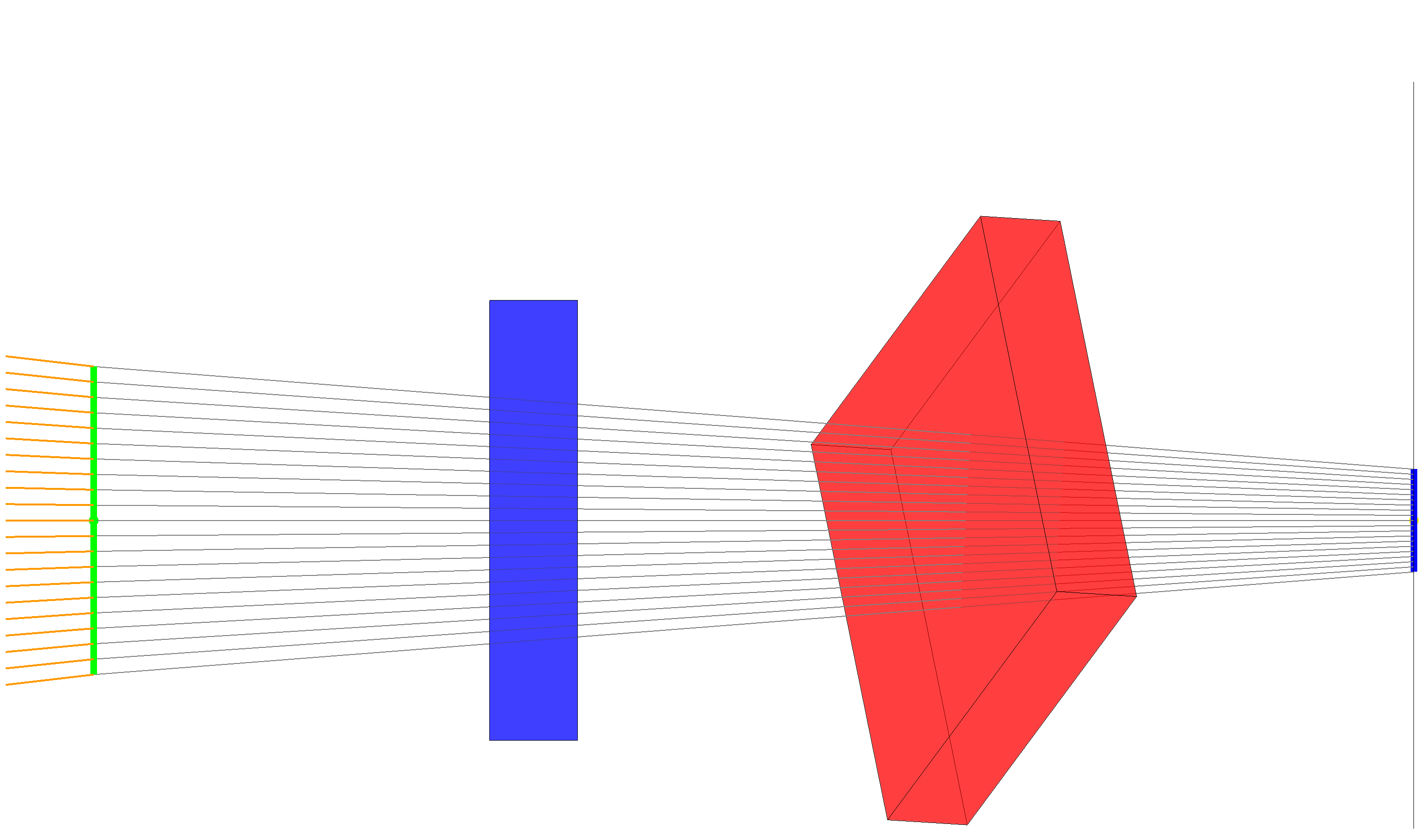}}{1}
              \draw[thick, dashed] (-6.35,0.6) -- (-6.35,3);
              \draw[thick, dashed]  (7.22,-0.45) -- (7.22,3);
              \draw[thick, dashed]  (7.2,-0.52) -- (-6.35,-0.52);
              \draw[thick, dashed]  (7.2,-1.02) -- (-6.35,-1.02);
              \draw[line width=0.8pt, bend left=40] (-6.3, 0.55) to (-6.2,0.3);
              \draw[line width=0.8pt, bend right=40] (-6.2,0.3) to (-6.15, 0.023);
              \draw[line width=0.8pt, bend right=40] (-6.15, 0.027) to (-6.2,-0.3);
              \draw[line width=0.8pt, bend left=40] (-6.2,-0.3) to (-6.3, -0.5);
              \draw[line width=0.8pt, bend left=40] (-6.3, -0.53) to (-6.25,-0.68);
              \draw[line width=0.8pt, bend right=40] (-6.25,-0.67) to (-6.15, -0.765);
              \draw[line width=0.8pt, bend right=40] (-6.15, -0.765) to (-6.25,-0.86);
              \draw[line width=0.8pt, bend left=40] (-6.25,-0.85) to (-6.3, -1);
              \draw[line width=0.8pt, bend right=30] (-6.4, 0.55) to (-6.6,-0.45);
              \draw[line width=0.8pt, bend left=40] (-6.6,-0.45) to (-6.7, -1);
              \draw[line width=0.8pt, bend left=40] (-6.7, -1) to (-6.6,-1.55);
              \draw[line width=0.8pt, bend right=30] (-6.6,-1.55) to (-6.4, -2.55);
              \draw[line width=0.8pt, bend right=20] (7.1, -0.53) to (7.05,-1.15);
              \draw[line width=0.8pt, bend left=50] (7.05,-1.15) to (6.95, -1.25);
              \draw[line width=0.8pt, bend left=40] (6.95, -1.25) to (7.05,-1.35);
              \draw[line width=0.8pt, bend right=40] (7.05,-1.35) to (7.1, -1.57);
              \draw[line width=0.5pt] [->] (-7,3.3) .. controls (-7,-1) .. (-6.75,-1);
              \draw[line width=0.5pt] [->] (-5.7,1.5) .. controls (-5.7,0.02) .. (-6.1,0.02);
              \draw[line width=0.5pt] [->] (-5.7,-3) .. controls (-5.7,-0.755) .. (-6.1,-0.755);
              \draw[line width=0.5pt] [->] (6.3,-3) .. controls (6.3,-1.24) .. (6.9,-1.24);
              \draw[->] (0.8,2.8) -- (7.15,2.8);
              \draw[->] (-0.8,2.8) -- (-6.3,2.8);
              \node at (0,2.8) {$|v_0-v_x|$};
              \node at (-6.6,3.5) {$2^{-2(d-2)}\ell_i$};
              \node at (-3.8,1.7) {$\tfrac{1}{3}\!\cdot\!2^{-2(d-2)}\ell_i = c_i\cdot|v_0-v_x| $};
              \node at (-5.3,-3.3) {$\tfrac{1}{6}\!\cdot\!2^{-2(d-2)}\ell_i$};
              \node at (6.1,-3.3) {$\tfrac{1}{3}\!\cdot\!2^{-2(d-2)}\ell_i$};
            \end{annotate}
            \subcaption{Intuition behind the choice of the constants $c_1$ and $c_2$, using the proportion of the side-length of $\tfrac{1}{3}(B_{T_{v_x}}-v_x)+v_x$ (the blue line on the right side of the picture) and $P_{y_0,y_x}^{v_0}(B_{T_{v_x}})$ (the green line), where $\ell_i=f_{n_1}$ for $i\in\{2,\dots,k+1\}$ and $\ell_i=\epsilon$ for $i>k+1$.}
              \end{subfigure}
        \begin{subfigure}{1\textwidth}\begin{annotate}{
              \includegraphics[width=1\linewidth]{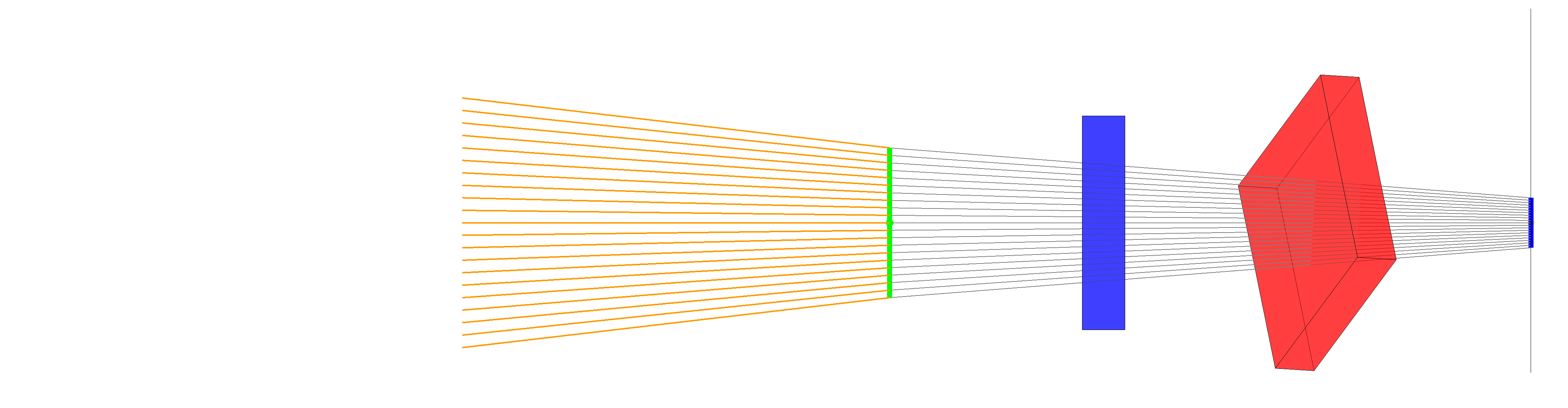}}{1}
              \draw[line width=0.6pt, bend right=20] (-3,1) to (-3.1,0);
              \draw[line width=0.6pt, bend left=30] (-3.1,0) to (-3.25,-0.2);
              \draw[line width=0.6pt, bend left=30] (-3.25,-0.2) to (-3.1,-0.4);
              \draw[line width=0.6pt, bend right=20] (-3.1,-0.4) to (-3,-1.35);
              \draw[<-] (-2.98,-2) -- (0.5,-2);
              \draw[->] (3.5,-2) -- (6.978,-2);
              \draw[thick] (-3,-1.8) -- (-3,-2.2);
              \draw[thick, dashed] (6.98,-0.4) -- (6.98,-2.2);
              \node at (2,-2) {$|y^{(1)}|+|v_0-v_x|$};
              \node at (-5.5,-0.2) {$c_i \cdot (|y^{(1)}|+|v_0-v_x|)  $};
            \end{annotate}
            \subcaption{An illustration of the choice of $y^{(i)}$ in $I_{y_0,y_x}$ for $i\in\{2,\dots,d\}$.}
              \end{subfigure}
         \caption{Illustration of the role of $c_1$ and $c_2$ in $\mathbb{R}^3$. In the above figures, $c_i$ equals $c_1$ for $i\in\{2,\dots,k+1\}$ and equals $c_2$ otherwise. The colour of the lines constructing $I_{y_0,y_x}$ are grey as before; similarly the colours for the boundary of $I_{y_0,y_x}$, $C_{y_0}$, $C_{y_x}$, $P_{y_0,y_x}^{v_0}(B_{T_{v_x}})$ and $\tfrac{1}{3}(B_{T_{v_x}}-v_x)+v_x$ are the same as in Figure \ref{fig:connector_of_y0yx}.}
              \label{fig:c1c2_for_connector}
\end{figure}
    To keep notation in the following calculations concise, we write $EF_{y_0,y_x}^{z_0,z_x}$ instead of $E_2^{(y_0,y_x)}  \cap F_{n_1}^{(0,z_0,y_0)}\cap F_{n_1}^{(x,z_x,y_x)}\cap E_{0,z_0}^b\cap E_{x,z_x}^b$. 
    \begin{align*}
    & \mathbb{P}_{0,x,z_0,z_x,y_0,y_x}\bigl(E_1^{(y_0,y_x)} \,|\, E_2^{(y_0,y_x)}  \cap F_{n_1}^{(0,z_0,y_0)}\cap F_{n_1}^{(x,z_x,y_x)}\cap E_{0,z_0}^b\cap E_{x,z_x}^b\bigr)\\
     &\leq \mathbb{E}_{0,x,z_0,z_x,y_0,y_x}\Big[  \mathbb{P}_{0,x,z_0,z_x,y_0,y_x}\big[ \text{ there exists no }  \textbf{y} \in \mathcal{X} \mbox{ such that } y\in B_{|x|^{\frac{\min\{d-k,k\}}{\alpha_k-k}-\epsilon}}(0)^c\!\!:\\
       &\hspace{83pt} C_y\cap C_{y_0} \neq \emptyset \text{ and }C_y\cap C_{y_x} \neq \emptyset
      \,\mid\, \mathfrak F_n^x \big]\mathbbm{1}_{ EF_{y_0,y_x}^{z_0,z_x}}\Big]\big/{\mathbb{P}_{0,x,z_0,z_x,y_0,y_x}}(EF_{y_0,y_x}^{z_0,z_x})\\
      &\leq \mathbb{E}_{0,x,z_0,z_x,y_0,y_x}\Big[  \mathbb{P}_{0,x,z_0,z_x,y_0,y_x}\big[ \text{ there exists no }  \textbf{y} \in \mathcal{X} \mbox{ such that } y\in I_{y_0,y_x}:\\
       &\hspace{85pt} C_y\cap C_{y_0} \neq \emptyset \text{ and }C_y\cap C_{y_x} \neq \emptyset
      \,\mid\, \mathfrak F_n^x \big]\mathbbm{1}_{ EF_{y_0,y_x}^{z_0,z_x}}\Big]\big/\mathbb{P}_{0,x,z_0,z_x,y_0,y_x}( EF_{y_0,y_x}^{z_0,z_x})\\
      &= \mathbb{E}_{0,x,z_0,z_x,y_0,y_x}\Big[\exp\Big(-u\!\!\!\!\smash{\int\limits_{I_{y_0,y_x} }}\!\! \mathbb{P}_C\big(C_y\cap C_{y_0} \neq \emptyset \text{ and }C_y\cap C_{y_x} \neq \emptyset\big)\,\text{d}\lambda (y)\Big)\mathbbm{1}_{EF_{y_0,y_x}^{z_0,z_x}} \Big]\\
      &\pushright{\big/\mathbb{P}_{0,x,z_0,z_x,y_0,y_x}(EF_{y_0,y_x}^{z_0,z_x}).}
    \end{align*}
    The last inequality follows since $I_{y_0,y_x}\subseteq B_{|x|^{\frac{\min\{d-k,k\}}{\alpha_k-k}-\epsilon}}(0)^c$.
    We can continue bounding the previous expression from above by
    \begin{align*}
        &\mathbb{E}_{0,x,z_0,z_x,y_0,y_x}\Bigl[\exp\Bigl(-u\!\!\!\!\smash{\int\limits_{I_{y_0,y_x} }} \mathbb{P}_C\bigl( C_y\cap \bigl(\tfrac{1}{3}(B_{T_{v_{x}}}-v_x)+v_x\bigr) \neq \emptyset\text{ and } \\
        &\hspace{125pt} D_y^{(k)}\geq 2(|v_x-v_0|+|y-v_0|)\bigr)\text{d}\lambda (y)\Bigr)\mathbbm{1}_{EF_{y_0,y_x}^{z_0,z_x}} \Bigr]\\
        &\pushright{\big/{\mathbb {P}_{0,x,z_0,z_x,y_0,y_x}}(EF_{y_0,y_x}^{z_0,z_x})}\\
        &=\mathbb{E}_{0,x,z_0,z_x,y_0,y_x}\Bigl[\exp\Bigl(-u\!\!\!\!\smash{\int\limits_{I_{y_0,y_x} }} \mathbb{P}_C(C_y\cap \tfrac{1}{3}(B_{T_{v_{x}}}-v_x)+v_x) \neq \emptyset \\
        &\hspace{170pt}\big|\, D_y^{(k)}\geq 2 (|v_x-v_0|+|y-v_0|)\bigr)\\
        &\hspace{160pt}\times\mathbb{P}\bigl(D_y^{(k)}\geq 2(|v_x-v_0|+|y-v_0|)\bigr)\,\text{d}\lambda (y)\Bigr)\mathbbm{1}_{EF_{y_0,y_x}^{z_0,z_x}} \Bigr]\\
        &\pushright{\big/{\mathbb{P}_{0,x,z_0,z_x,y_0,y_x}}(EF_{y_0,y_x}^{z_0,z_x}).}
    \end{align*} 
    In order to bound the integrand we use translation and rotation invariance of the Lebesgue measure, Potter bounds and the size of the set of orientations of $C_y$ that results in an intersection with $\tfrac{1}{3}(B_{T_{v_{x}}}-v_x)+v_x) $ which yields
    \begin{align*}
        \mathbb{P}_C(\rot_{y_0,y_x}^{-1}(C_y-v_0)\cap  \rot_{y_0,y_x}^{-1}\bigl(\tfrac{1}{3}&(B_{T_{v_{x}}}-v_x)+v_x-v_0\bigr)  \neq \emptyset \big|\, D_y^{(k)}\geq 2(|v_x-v_0|+|y|)\bigr)\\
        &\times\mathbb{P}\bigl(D_y^{(k)}\geq 2 (|v_x-v_0|+|y|)\bigr)\\
        &\quad\quad\geq c  \frac{f_{n_1}^{\min\{d-k,k\}}}{(|v_x-v_0|+|y|)^{d-k}}\bigl( |v_x-v_0|+|y|\bigr)^{-\alpha_k-\varepsilon}
    \end{align*}
    and use this lower bound going forward. Using now that $|v_x-v_0|\in(|x|/4,2|x|)$, ${|y|>cf_{n_1+1}}$, $|x|/8=f_{n_1}$ and the norm equivalence in $\mathbb{R}^d$  we can bound the last term from below by
    \begin{equation*}
        c\frac{|x|^{\min\{d-k,k\}}}{|y|_1^{d-k}}|y|_1^{-\alpha_k-\varepsilon}= c|x|^{\min\{d-k,k\}}|y|_1^{-\alpha_k-\varepsilon-d+k}
    \end{equation*}
    where $|\cdot|_1$ is the $1$-norm. Putting it all together, we can therefore bound 
    \begin{equation*}
        \mathbb{P}_{0,x,z_0,z_x,y_0,y_x}\bigl(E_1^{(y_0,y_x)} \,|\, E_2^{(y_0,y_x)}  \cap F_{n_1}^{(0,z_0,y_0)}\cap F_{n_1}^{(x,z_x,y_x)}\cap E_{0,z_0}^b\cap E_{x,z_x}^b\bigr)
    \end{equation*}
    from above by
    \begin{equation}\label{eq:upperbound_no_good_connector1}
        \begin{split}\mathbb{E}_{0,x,z_0,z_x,y_0,y_x}&\Bigl[\exp\Bigl(-uc\!\!\smash{\int\limits_{I_{0,x} }}|x|^{\min\{d-k,k\}}|y|_1^{-\alpha_k-\varepsilon-d+k}\,\text{d}\lambda (x)\Bigr)\mathbbm{1}_{E_{y_0,y_x}^{z_0,z_x}} \Bigr]\\
        &\hspace{180pt}\big/{\mathbb P_{0,x,z_0,z_x,y_0,y_x}}(E_{y_0,y_x}^{z_0,z_x}).
        \end{split}
    \end{equation}
    We focus now at the integral and get by using the definition of $I_{0,x}$ 
    \begin{align*}
        &\int\limits_{I_{0,x}} |x|^{\min\{d-k,k\}}|y|_1^{-\alpha_k-\varepsilon-d+k}\,\text{d}\lambda (x) \\
        &\geq\!\! \int_{|x|^{\frac{\min\{d-k,k\}}{\alpha_k-k}-\epsilon}}^{\infty}\text{d}y^{(1)}\!\!\int_{-c_1(y^{(1)}+|v_0-v_x|)}^{c_1(y^{(1)}+|v_0-v_x|)}\!\!\!\!\text{d}y^{(2)}\!\dots\!\!\int_{-c_1(y^{(1)}+|v_0-v_x|)}^{c_1(y^{(1)}+|v_0-v_x|)}\!\!\!\! \text{d}y^{(k+1)}\!\!\int_{-c_2(y^{(1)}+|v_0-v_x|)}^{c_2(y^{(1)}+|v_0-v_x|)}\!\!\!\!\text{d}y^{(k+2)}\\
        &\hspace{130pt}\dots\int_{-c_2(y^{(1)}+|v_0-v_x|)}^ {c_2(y^{(1)}+|v_0-v_x|)}     
         |x|^{\min\{d-k,k\}}|y|_1^{-\alpha_k-\varepsilon-d+k} \text{d}y^{(d)},
    \end{align*}
    where we highlight the change of integration range when going from $k+1$ to $k+2$.
    Using $y^{(1)}+|v_0-v_x|\geq \tfrac{1}{2}y^{(1)}$ we can bound this from below by
    \begin{align}
        &\int_{|x|^{\frac{\min\{d-k,k\}}{\alpha_k-k}-\epsilon}}^{\infty}\text{d}y^{(1)}\int_{-\frac{c_1}{2}y^{(1)})}^{\frac{c_1}{2}y^{(1)}}\text{d}y^{(2)}\dots\int_{-\frac{c_1}{2}y^{(1)})}^{\frac{c_1}{2}y^{(1)}}\text{d}y^{(k+1)}\int_{-\frac{c_2}{2}y^{(1)}}^{\frac{c_2}{2}y^{(1)}}\text{d}y^{(k+2)}\nonumber\\
        &\hspace{180pt}\dots\int_{-\frac{c_2}{2}y^{(1)}}^ {\frac{c_2}{2}y^{(1)}} |x|^{\min\{d-k,k\}}|y|_1^{-\alpha_k-\varepsilon-d+k}\text{d}y^{(d)} \nonumber\\ 
        &\geq c \int_{|x|^{\frac{\min\{d-k,k\}}{\alpha_k-k}-\epsilon}}^{\infty} |x|^{\min\{d-k,k\}}(|y^{(1)}|+k|\tfrac{c_1}{2}y^{(1)}|+(d-k-1)|\tfrac{c_2}{2}y^{(1)}|)^{-\alpha_k-\varepsilon+k-1}\text{d}y^{(1)}\nonumber\\      
        &\geq c\int_{|x|^{\frac{\min\{d-k,k\}}{\alpha_k-k}-\epsilon}}^{\infty} |x|^{\min\{d-k,k\}}|y^{(1)}|^{-\alpha_k-\varepsilon+k-1} \text{d}y^{(1)}\nonumber\\  
        &\geq c|x|^{\min\{d-k,k\}}\Bigl(|x|^{\frac{\min\{d-k,k\}}{\alpha-k}-\epsilon}\Bigr)^{-\alpha_k-\varepsilon+k}.\label{eq:upperbound_no_good_connector2}
    \end{align}
    In the first inequality we integrated over the last $d-1$ coordinates of $y$. The second inequality is given by using definitions of $c_1$ and $c_2$. Integrating over the first coordinate of $y$ then gives the final inequality.

    Continuing the calculation of \eqref{eq:upperbound_no_good_connector1}, substituting in the bound from \eqref{eq:upperbound_no_good_connector2}, we get
    \begin{align*}
        \mathbb{E}_{0,x,z_0,z_x,y_0,y_x}\Bigl[&\exp\Bigl(-uc|x|^{\min\{d-k,k\}}\Bigl(|x|^{\frac{\min\{d-k,k\}}{\alpha-k}-\epsilon}\Bigr)^{-\alpha_k-\varepsilon+k}\Bigr)\mathbbm{1}_{E_{y_0,y_x}^{z_0,z_x}} \Bigr]\\
        &\pushright{\big/{\mathbb P_{0,x,z_0,z_x,y_0,y_x}}(E_{y_0,y_x}^{z_0,z_x})} \\
        &\leq \exp\Bigl(-uc|x|^{-\varepsilon\frac{\min\{d-k,k\}}{\alpha_k-k}+\epsilon(\alpha_k +\varepsilon -k) }\Bigr)\\
        &\leq  \exp\bigl(  -u c |x|^{\epsilon(\alpha_k -k)/2}\bigr),
    \end{align*}
    where the last inequality follows using that for $0<\varepsilon<(\alpha_k-k)^2\epsilon\big/\big(2\min\{d-k,k\}\big)$ (recall that we can choose $\varepsilon$ appearing in the Potter bounds arbitrarily small) we have
    \begin{equation*}
        -\varepsilon\frac{\min\{d-k,k\}}{\alpha_k-k}+\epsilon(\alpha_k +\varepsilon -k) > \frac{\epsilon}{2}(\alpha_k-k) >0.
    \end{equation*}
\end{proof}
With this, we have proven that the first term of \eqref{eq:upperbound_probability_no_connector} converges to zero as $|x|$ becomes large. It remains to prove the same for the second term, which we do with the following lemma.
\begin{lemma}\label{Lemma:four}
    We have
    \begin{equation*}
         \mathbb{P}_{0,x,z_0,z_x,y_0,y_x}\bigl( E_3^{(y_0,y_x)}  \cap F_{n_1}^{(0,z_0,y_0)}\cap F_{n_1}^{(x,z_x,y_x)}\cap E_{0,z_0}^b\cap E_{x,z_x}^b\bigr) \leq \exp(- uc |x|^{\epsilon(\alpha_k-k)-\delta})
    \end{equation*}
    for $0<\delta<\epsilon(\alpha_k-k)$ and $|x|$ large enough. 
\end{lemma}
\begin{proof}
    Let $\tilde{E}_3^{(y_0,y_x)}$ be the event defined in such a way that
    \begin{equation*}
        \tilde E_3^{(y_0,y_x)}  \cap F_{n_1}^{(0,z_0,y_0)}\cap F_{n_1}^{(x,z_x,y_x)}\cap E_{0,z_0}^b\cap E_{x,z_x}^b= E_3^{(y_0,y_x)}  \cap F_{n_1-1}^{(0,z_0)}\cap F_{n_1-1}^{(x,z_x)}\cap E_{0,z_0}^b\cap E_{x,z_x}^b 
    \end{equation*}
    holds, i.e.\ the vertices $\mathbf{y}_0$ and $\mathbf{y}_x$ are good with respect to each other and  $\mathbf{y}_0$ (resp.\ $\mathbf{y}_x$) can be used to extend the path of length $n_1-1$ starting in $\mathbf{z}_0$ (resp.\ $\mathbf{z}_x$) so that the extended path implies the event $F_{n_1}^{(0,z_0,y_0)}$ (resp.\ $F_{n_1}^{(x,z_x,y_x)}$). We have
    \begin{align*}
        \mathbb{P}_{0,x,z_0,z_x,y_0,y_x}&\bigl(E_3^{(y_0,y_x)}  \cap F_{n_1}^{(0,z_0,y_0)}\cap F_{n_1}^{(x,z_x,y_x)}\cap E_{0,z_0}^b\cap E_{x,z_x}^b\bigr) \\&=\mathbb{P}_{0,x,z_0,z_x,y_0,y_x}\bigl( \tilde E_3^{(y_0,y_x)}  \cap F_{n_1-1}^{(0,z_0)}\cap F_{n_1-1}^{(x,z_x)}\cap E_{0,z_0}^b\cap E_{x,z_x}^b\bigr)\\
        &\leq \mathbb{P}_{0,x,z_0,z_x,y_0,y_x}\bigl(\tilde E_3^{(y_0,y_x)}  \,|\, F_{n_1-1}^{(0,z_0)}\cap F_{n_1-1}^{(x,z_x)}\cap E_{0,z_0}^b\cap E_{x,z_x}^b\bigr).
    \end{align*}
    To bound this probability from above recall that in general, for vertices of paths we consider, for every $n\in\mathbb{N}$ and suitable $z\in \mathbb{R}^d$, i.e.\ $z\in O_{n+1}(\mathbf{x}_{n})$,  using a similar argument as for \eqref{eq:TMM_key_step}, we can obtain the following inequality 
    \begin{equation}\label{eq:for_thinning}
        \begin{split}
        \mathbb{P}_C\bigl(  (z+C)\cap   Q^*_{f_n}(x_n,x)\neq \emptyset\,|\, D_{C}^{_{(k)}} \geq 2^{2(d-1){+\epsilon}}f_{_{n+1}} \bigr)  &\mathbb{P}_C(D_{C}^{_{(k)}} \geq2^{2(d-1){+\epsilon}}f_{_{n+1}})\\
        &\geq c \frac{f_n^{\min\{d-k,k\}}}{f_{n+1}^{d-k}} f_{n+1}^{-\alpha_k+\varepsilon}.
        \end{split}
    \end{equation}
    Setting $n=n_1-1$, the right hand side of this inequality is a lower bound for the probability that a vertex with given location in $O_{n+1}(\mathbf{x}_{n})$ is the $n_1$th vertex in the growing path as required for $\mathbf{y}_0$ (resp.\ $\mathbf{y}_x$) to satisfy $F_{n_1}^{(0,z_0,y_0)}$ (resp.\ $F_{n_1}^{(x,z_x,y_x)}$) under the condition $ F_{n_1-1}^{(0,z_0)}\cap F_{n_1-1}^{(x,z_x)}\cap E_{0,z_0}^b\cap E_{x,z_x}^b$. In addition to that note that $\mathbf{y}_0$ (resp.\ $\mathbf{y}_x$) has location in $O_0:=B_{3|x|/8}(0)\setminus B_{|x|/8}(0)$ (resp.\ $O_x:=B_{3|x|/8}(x)\setminus B_{|x|/8}(x)$). We therefore have that $y_0$ and $y_x$ lie in disjoint areas. 
    
    Writing from here on $ F_{n_1-1}^{(0,z_0)}\cap F_{n_1-1}^{(x,z_x)}\cap E_{0,z_0}^b\cap E_{x,z_x}^b$ as $G_{n_1-1}^{(0,x,z_0,z_x)}$ and using properties of the Poisson point process we calculate 
    \begin{align*}
        &\mathbb{P}_{0,x,z_0,z_x,y_0,y_x}\bigl(E_3^{(y_0,y_x)}  \,|\, F_{n_1}^{(0,z_0,y_0)}\cap F_{n_1}^{(x,z_x,y_x)}\cap E_{0,z_0}^b\cap E_{x,z_x}^b\bigr) \\ 
        &= \mathbb{E}_{0,x,z_0,z_x,y_0,y_x}\Big[  \mathbb{P}_{0,x,z_0,z_x,y_0,y_x}\bigl[ \not\exists  \mathbf{v},\mathbf{w} \in \mathcal{X} \text{ that extend } F_{n_1-1}^{(0,z_0)} \text{ and } F_{n_1-1}^{(x,z_x)}\text {such that } \mathbf{v}\\
       &\hspace{110pt} \text{ and } \mathbf{w} \text{ are good with respect } \text{to each other.}\,|\, \mathfrak F_{n_1-1}^x \bigr]\mathbbm{1}_{ G_{n_1-1}^{(0,x,z_0,z_x)}}\Big]\\
       &\pushright{\big/{\mathbb{P}_{0,x,z_0,z_x,y_0,y_x}}(G_{n_1-1}^{(0,x,z_0,z_x)})}\\
      &\leq \mathbb{E}_{0,x,z_0,z_x,y_0,y_x}\Bigl[ \exp\Bigl( -uc \!\!\!\ \int\limits_S\!\! \text{d}\mathbb{P}_{p_{v}} \!\!\int\limits_{O_0}\!\!\text{d}\lambda(w)\!\!\!\!\!\!\! \!\!\!\!\int\limits_{O_x\cap A(w,n_1)} \!\!\!\!\!\!\!\!\!\!\! \mathbb{P}_{v,w,p_v} \bigl( \measuredangle(p_{w}^{(i)},p_{v}^{(i)})\leq \tfrac{\epsilon}{\log|x|},\,\forall i\in\{1,\dots,d\} \bigr)\\
      &\pushright{\!\!\!\smash{\frac{f_{n_1-1}^{2\min\{d-k,k\}}}{f_{n_1}^{2(d-k)}} f_{n_1}^{-2(\alpha_k-\varepsilon)}} \text{d} \lambda(v)
         \Bigr)\mathbbm{1}_{ G_{n_1-1}^{(0,x,z_0,z_x)}}\Bigr]}\\
    &\pushright{\big/{\mathbb{P}_{0,x,z_0,z_x,y_0,y_x}}(G_{n_1-1}^{(0,x,z_0,z_x)}),}
    \end{align*}
    where $S$ is the set of all possible orientations of $C$ (see below for a more precise definition) and we used \eqref{eq:for_thinning} to thin out the Poisson point process in the inequality. Note that the bound in \eqref{eq:for_thinning} only holds for pairs of vertices for which $v\in O_{n_1}(\mathbf{x}_{n_1-1})$ (and similarly for $w$). Since we are restricting ourselves to $O_0$ and $O_x$, this assumption is satisfied. In particular, since $O_0$ and $O_x$ are bounded, there exists a constant $c>0$ for which \eqref{eq:for_thinning} holds for all vertices with locations inside $O_0$ and $O_x$. Furthermore, in line with our notation, $p_v$ in the index of $\mathbb{P}$ signifies that we conditioned on $p_v$.
    
    To continue, define 
    \begin{equation*}
        S:=\bigtimes_{i=1}^{d-1}\mathbb{S}^{d-i}
    \end{equation*}
    to be the set of all possible orientations of $C$. Here, $\mathbb{S}^{d-i}$ is the set of possible orientations for $p_C^{(i)}$, which is the orientation of the $i$th side of $C$, for $i\in\{1,\dots,d-1\}$, relative to all preceding orientations. Note that the orientation of $p_C^{(d)}$ is completely determined (up to mirroring) when the orientations of $p_C^{(1)},\dots,p_C^{(d-1)}$ are known; also observe that given all preceding orientations, $p_C^{(i)}$ is distributed uniformly on $\mathbb{S}^{d-i}$. We write $\mathbb{P}_{p_C}$ for the law of the vector $(p_C^{(1)},\dots,p_C^{(d)})$ and $\mathbb{P}_{p_v}$ when $C$ is attached to a specific location $v\in\mathscr{P}$.

    To bound $\mathbb{P}_{0,x}\bigl( \tilde E_3\, | \,F_{n_1-1}^{(0)}\cap F_{n_1-1}^{(x)}\cap E_0^b\cap E_x^b\bigr) $ from above we calculate
    \begin{align*}
        &\smash{\int\limits_S\text{d} \mathbb{P}_{p_{v}} \int\limits_{O_0}\text{d}\lambda(w)\!\!\!\!\!\!\!\!\!\!\!\!\int\limits_{O_x\cap A(w,n_1)}} \!\!\!\!\!\! \frac{f_{n_1-1}^{2\min\{d-k,k\}}}{f_{n_1}^{2(d-k)}} f_{n_1}^{-2(\alpha_k-\varepsilon)} \\
       &\hspace{120pt}\times\mathbb{P}_{v,w,p_v} \bigl( \measuredangle(p_{w}^{(i)},p_{v}^{(i)})\leq \tfrac{\epsilon}{\log|x|},\,\forall i\in\{1,\dots,d\} \bigr)\text{d} \lambda(v)   \\
       &\leq  c\int\limits_S\text{d} \mathbb{P}_{p_{v}} \int\limits_{O_0}\text{d}\lambda(w)\!\!\!\!\!\!\!\!\!\int\limits_{O_x\cap A(w,n_1)} \!\!\!\!\!\! \frac{f_{n_1-1}^{2\min\{d-k,k\}}}{f_{n_1}^{2(d-k)}} f_{n_1}^{-2(\alpha_k-\varepsilon)}\log(|x|)^{-(d-1)!}\text{d} \lambda(v)  \\
       &\leq cf_{n_1}^{2d}\frac{f_{n_1-1}^{2\min\{d-k,k\}}}{f_{n_1}^{2(d-k)}} f_{n_1}^{-2(\alpha_k-\varepsilon)} \log(|x|)^{-(d-1)!}.
    \end{align*}
    The first inequality holds since the set of allowed orientations for $p_{w}^{(i)}$ for fixed $p_{v}^{(i)}$ has volume of order $\log(|x|)^{-(d-1)!}$, which can be shown by induction. Intuitively, looking at the sets that are given for the orientations of the sides of $C$, we first consider the longest side and determine the rotational area which is allowed so that it has suitable orientation; this is a set of size proportional to $\log(|x|)^{-d+1}$. Fixing this first orientation, it remains to consider the remaining $d-1$ sides. For the orientation of the second side we have a restriction of the rotational area which is proportional to $\log(|x|)^{-(d-2)}$ and so on, until we are looking at the second to last side, where we have a rotational area proportional to $\log(|x|)^{-1}$ left. 
    Altogether this leads to the exponent $-(d-1)!$. For the final inequality above we use that the volume of $O_0$ resp.\ $O_x\cap A(w,n_1)$ for fixed $w\in\mathbb{R}^d$ is of order $f_{n_1}^d$ and that $\mathbb{P}_{p_v}$ is uniform on $S$.  

    Using $0<\varepsilon<(\alpha_k-k)^2\epsilon\big/\big(2\min\{d-k,k\}\big)$ as at the end of the proof of Lemma~\ref{Lemma:three}, the last expression can be bounded from below by
    \begin{equation*}
        f_{n_1}^{\epsilon(\alpha_k-k)} \log(|x|)^{-(d-1)!}\geq c|x|^{\epsilon(\alpha_k-k)-\delta}, 
    \end{equation*}
    with $\delta>0$ small enough such for the exponent of $|x|$ to be positive.

    Together we get 
    \begin{equation*}
        \mathbb{P}_{0,x,z_0,z_x,y_0,y_x}\bigl( E_3^{(y_0,y_x)}  \cap F_{n_1}^{(0,z_0,y_0)}\cap F_{n_1}^{(x,z_x,y_x)}\cap E_{0,z_0}^b\cap E_{x,z_x}^b\bigr) \leq \exp(- uc |x|^{\epsilon(\alpha_k-k)-\delta}). 
    \end{equation*}
\end{proof}
Note that the upper bounds from Lemma~\ref{Lemma:three} and Lemma~\ref{Lemma:four} do not depend on $z_0,z_x,y_0$ and $y_x$. Consequently, integrating over these locations and using these uniform bounds gives an upper bound for the last term of \eqref{eq:upperbound_key_calculation}. This implies that the chemical distance for fixed $k\in M$ is given by $\frac{2+\delta}{\log\bigl( \frac{\min\{d-k,k\}}{\alpha_k-k}\bigr)} \log\log|x|$ with high probability as $|x|\rightarrow \infty.$

Finally, we get the claimed upper bound as follows. Recall that $k$ is such that ${\alpha_k\in(k,\min\{2k,d\})}$, i.e.\ $k\in M$. Furthermore, we have proven that the upper bound holds for the ``smallest'' Poisson Boolean model with convex grains. This gives us also an upper bound for the chemical distance for every other Poisson Boolean base model with $\alpha_k\in(k,\min\{d,2k\})$, for any and therefore every $k\in M$.
Together, we get that $\dist(\mathbf{0},\mathbf{x})$ is, under the condition that $0,x\in\mathscr{P}$ and $\mathbf{0}\leftrightarrow \mathbf{x}$, for every $\delta >0$, bounded from above by
\begin{equation*}
    \min\left\{\frac{2+\delta}{\log\bigl( \frac{\min\{d-k,k\}}{\alpha_k-k}\bigr)} \log\log|x| \,:\, k\in M\right\}= \frac{2+\delta}{\log\bigl( \frac{\min\{d-\kappa,\kappa\}}{\alpha_{\kappa}-\kappa}\bigr)} \log\log|x|,
\end{equation*}
with high probability as $|x|\rightarrow \infty$.

\begin{rem}\label{rem:schnell}
For the case where there exists some $k\in\{1,\dots,d\}$ such that $\alpha_k\leq k $ this model dominates (in the standard, almost sure coupling sense; see Remark \ref{rem:langsam}) all other models where we replace $\alpha_k$ by $k+\rho$ for $\rho>0$ in the ``smallest'' box model as in the proof of the upper bound. Looking at the smaller model we have already shown that the upper bound for the chemical distance is given by $(2+\delta)/{\log\bigl( \tfrac{\min\{d-k,k\}}{\rho}\bigr)} \log\log|x|$. With $\rho\downarrow 0$ we get that the factor $(2+\delta)/{\log\bigl( \tfrac{\min\{d-k,k\}}{\rho}\bigr)}$ is getting arbitrary small. With this we get that $\dist(\mathbf{x},\mathbf{y})$ grows in the bigger model faster than $c\log\log|x-y|$ for all $c>0$. 
\end{rem}

\section{Examples}\label{Sect:examples}
In this section we consider multiple examples of models for which Theorem \ref{Theorem} holds. 
\subsection*{Ellipsoids with long and short axes}
We consider the case of ellipsoids where we have $m$ short axes of length one and $d-m$ long axes of length $R$ for $m\in\{1,\dots,d-1\}$ and $R$ regularly varying with index $-\alpha$ for $\alpha>0$. It was shown in \cite{Gracar2024a} that the model is robust if $\alpha < \min\{2(d-m),d\}.$ We get by definition that $\alpha_k=\alpha_{d-m}$ for all $k\in\{1,\dots,d-m\}$ and Theorem \ref{Theorem} gives us
\begin{equation*}
    \lim\limits_{|x-y|\rightarrow \infty}\mathbb{P}_{x,y}\Bigl( \frac{2-\delta} {\log\bigl( \frac{\min\{m, d-m \}}{\alpha_{d-m}-d+m}\bigr) }\, \leq \,\frac{\dist(\mathbf{x},\mathbf{y})}{\log\log|x-y|} \, \leq\,  \frac{2+\delta} {\log\bigl( \frac{\min\{m, d-m \}}{\alpha_{d-m}-d+m} \bigr)}\, \, \Big|\, \, x\leftrightarrow y \Bigr) =1,
\end{equation*}
for $x,y\in\mathscr{P}$ and small $\delta>0$. 
In particular this theorem improves both the upper and lower bounds shown in \cite{Hilario2021}, where Hilário and Ungaretti considered the special case, namely $d=2$ and $m=1$ with ellipses such that the diameters fulfil $\mathbb{P}(D^{_{(1)}} \geq r) =cr^{-\alpha}$ for $c>0,\, \alpha\in (1,2)$ and $D^{_{(2)}}=1$ almost surely.
\subsection*{Ellipsoids with independent axes}
We consider the model of random ellipsoids generated via independent random radii $R_1,\dots,R_d$ with regularly varying tail indices $-\beta_1,\dots,-\beta_d$ for $\beta_d\geq\dots\geq\beta_1>0$. Using independence of these radii and the definition of the diameters we have that the $k$th diameter $D^{_{(k)}}$ is regularly varying with tail index $\smash{\sum_{i=1}^k \beta_i}$, which can be obtained by using properties of regularly varying functions. In \cite{Gracar2024a} it is shown that the model is robust if there exists $1\leq i\leq d$ such that $\beta _i <2$. In addition to that we assume that $\min_{1\leq i\leq d} \beta_i >1$ to guarantee that $\alpha_k> k$ for all $1\leq k\leq d$. Define now as in Theorem \ref{Theorem} $M:=\bigr\{s\in\{1,\dots,d-1\}:\, \sum_{i=1}^s \beta_i\in(s,\min\{2s,d\})\bigr\}$ and recall the definition of $\kappa$. We get then that
\begin{equation*}
    \mathbb{P}_{x,y}\Biggl( \frac{2-\delta} {\log \frac{\min\{\kappa, d-\kappa \}}{-\kappa+\sum_{i=1}^{\kappa}\beta_i} }\, \leq \,\frac{\dist(\mathbf{x},\mathbf{y})}{\log\log|x-y|} \, \leq\,  \frac{2+\delta} {\log \frac{\min\{\kappa, d-\kappa \}}{-\kappa+\sum_{i=1}^{\kappa}\beta_i} }\, \, \Bigg|\, \, x\leftrightarrow y \Biggr) \overset{|x-y|\rightarrow\infty}{\longrightarrow}1.
\end{equation*}
for $x,y\in\mathscr{P}$ while $\delta>0$. 
\subsection*{Ellipsoids with strongly dependent axes}

    Consider the model where the ellipsoids have axes of length $U^{-\beta_k}$ with $k\in\{1,\dots,d\}$, $U$ uniformly distributed on $(0,1)$, and $\beta_d\geq ..\geq \beta_1\geq 0$. For $x>1$ we get that
    \begin{equation*}
        \mathbb{P}(U^{-\beta_k} \geq x) = \mathbb{P}( U < x^{-1/\beta_k})= x^{-\frac{1}{\beta_k}},
    \end{equation*}
    i.e.\ the diameter $D^{_{(k)}}$ has regularly varying tail distribution with index $-1/\beta_{d-k+1}$ for $k\in\{1,\dots,d\}$.  
    In \cite{Gracar2024a} it was shown that the model is robust if $\beta_{d-k+1} > \max\{\frac{1}{2k},\frac{1}{d}\}$. Setting $M:= \bigl\{s\in\{1,\dots,d-1\}:\, \beta_{d-k+1}^{-1} \in(k,\min\{2k,d\})\bigr\}$ and $\kappa$ as in Theorem \ref{Theorem}, we get
    \begin{equation*}
    \mathbb{P}_{x,y}\Biggl( \frac{2-\delta} {\log \frac{\min\{\kappa, d-\kappa \}}{\beta_{d-\kappa+1}^{-1}-\kappa} }\, \leq \,\frac{\dist(\mathbf{x},\mathbf{y})}{\log\log|x-y|} \, \leq\,  \frac{2+\delta} {\log \frac{\min\{\kappa, d-\kappa \}}{\beta_{d-\kappa+1}^{-1}-\kappa} }\, \, \Bigg|\, \, x\leftrightarrow y \Biggr) \overset{|x-y|\rightarrow\infty}{\longrightarrow}1,
\end{equation*}
for $x,y\in\mathscr{P}$ while $\delta>0$. 
\subsection*{Random triangles}
Let $R>1$ be random with regularly varying tail of index $-\alpha$, for $\alpha >0$. Take the \emph{right} triangle $K\subset 
\mathbb R^2$ such that $R$ is the length of the hypotenuse and $\mathrm{Vol}(K)=\frac14 R^{1+\beta}$, for some $\beta\in(0,1)$. Note that this describes the triangle uniquely up to symmetries. Now choose the origin uniformly as one of the vertices of $K$
and let $C$ be the random set obtained by a uniform rotation around this point. We have therefore that $D^{_{(1)}}$ is the length of the hypotenuse, i.e.\ $D^{_{(1)}}=R$. Moreover $D^{_{(2)}}=\frac{1}{2}R^{\beta}$. In \cite{Gracar2024a} it was shown that the model is robust if $\alpha <2$. Considering the case $\alpha\in(1,2)$ to ensure the model is sparse we get from Theorem \ref{Theorem}
    \begin{equation*}
    \mathbb{P}_{x,y}\Biggl( \frac{2-\delta} {\log \frac{1}{\alpha-1} }\, \leq \,\frac{\dist(\mathbf{x},\mathbf{y})}{\log\log|x-y|} \, \leq\,  \frac{2+\delta} {\log \frac{1}{\alpha-1} }\, \, \Bigg|\, \, x\leftrightarrow y \Biggr) \overset{|x-y|\rightarrow\infty}{\longrightarrow}1.
\end{equation*}
for $x,y\in\mathscr{P}$ and $\delta>0$ small. 

\bigskip

\noindent\textbf{Acknowledgement:} We would like to thank Peter Mörters for the fruitful discussions that led to this work.
%

\bibliographystyle{plain}
\bibliography{library.bib}

\begin{thebibliography}{10}

\bibitem{Antal1996}
Peter Antal and Agoston Pisztora.
\newblock On the chemical distance for supercritical {B}ernoulli percolation.
\newblock {\em The Annals of Probability}, 24(2):1036--1048, 1996.

\bibitem{Cerny2012}
Ji{\v r}{\'\i} {\v C}ern{\'y} and Serguei Popov.
\newblock On the internal distance in the interlacement set.
\newblock {\em Electronic Journal of Probability}, 17, 2012.

\bibitem{Chebunin2024}
Mikhail Chebunin and G{\"u}nter Last.
\newblock On the uniqueness of the infinite cluster and the cluster density in the poisson driven random connection model.
\newblock {\em arXiv}, 2403.17762, 2024.

\bibitem{Dereich2012}
Steffen Dereich, Christian M{\"o}nch, and Peter M{\"o}rters.
\newblock Typical distances in ultrasmall random networks.
\newblock {\em Advances in Applied Probability}, 44(2):583--601, 2012.

\bibitem{Drewitz2014}
Alexander Drewitz, Bal{\'a}zs R{\'a}th, and Art{\"e}m Sapozhnikov.
\newblock On chemical distances and shape theorems in percolation models with long-range correlations.
\newblock {\em Journal of Mathematical Physics}, 55(8):083307, 30, 2014.

\bibitem{Garet2004}
Olivier Garet and R{\'e}gine Marchand.
\newblock Asymptotic shape for the chemical distance and first-passage percolation on the infinite bernoulli cluster.
\newblock {\em ESAIM: Probability and Statistics}, 8:169--199, 2004.

\bibitem{Garet2007}
Olivier Garet and R{\'e}gine Marchand.
\newblock {Large deviations for the chemical distance in supercritical Bernoulli percolation}.
\newblock {\em The Annals of Probability}, 35(3):833 -- 866, 2007.

\bibitem{Gouere2008}
Jean-Baptiste Gou{\'e}r{\'e}.
\newblock Subcritical regimes in the {P}oisson {B}oolean model of continuum percolation.
\newblock {\em The Annals of Probability}, 36(4):1209--1220, 2008.

\bibitem{Gouere2023}
Jean-Baptiste Gou{\'e}r{\'e} and Florestan Lab{\'e}y.
\newblock Percolation in the boolean model with convex grains in high dimension.
\newblock {\em Electronic Journal of Probability}, 28, 2023.

\bibitem{Gracar2022a}
Peter Gracar, Arne Grauer, and Peter M{\"o}rters.
\newblock Chemical distance in geometric random graphs with long edges and scale-free degree distribution.
\newblock {\em Communications in Mathematical Physics}, 2022.

\bibitem{Gracar2024a}
Peter Gracar, Marilyn Korfhage, and Peter M{\"o}rters.
\newblock Robustness in the poisson boolean model with convex grains.
\newblock {\em arXiv}, 2410.13366, 2024.

\bibitem{Gracar2022c}
Peter Gracar, Lukas L{\"u}chtrath, and Christian M{\"o}nch.
\newblock Finiteness of the percolation threshold for inhomogeneous long-range models in one dimension.
\newblock {\em arXiv}, 2203.11966, 2022.

\bibitem{Grimmett1990}
Geoffrey~R. Grimmett and John~Martin Marstrand.
\newblock The supercritical phase of percolation is well behaved.
\newblock {\em Proceedings of the Royal Society A}, 430(1879):439--457, 1990.

\bibitem{Hall1985}
Peter Hall.
\newblock On continuum percolation.
\newblock {\em The Annals of Probability}, 13(4):1250--1266, 1985.

\bibitem{Hao2023}
Nannan Hao and Markus Heydenreich.
\newblock Graph distances in scale-free percolation: the logarithmic case.
\newblock {\em Journal of Applied Probability}, 60(1):295--313, 2023.

\bibitem{Hilario2021}
Marcelo Hil{\'a}rio and Daniel Ungaretti.
\newblock Euclidean and chemical distances in ellipses percolation.
\newblock {\em arXiv}, 2103.09786, 2021.

\bibitem{Hirsch2020}
Christian Hirsch and Christian M{\"o}nch.
\newblock Distances and large deviations in the spatial preferential attachment model.
\newblock {\em Bernoulli}, 26(2):927--947, 2020.

\bibitem{Hug2023}
Daniel Hug, G{\"u}nter Last, and Wolfgang Weil.
\newblock Boolean models.
\newblock {\em arXiv}, 2308.05861, 2023.

\bibitem{Kulik2020}
Rafal Kulik and Philippe Soulier.
\newblock {\em Heavy-Tailed Time Series}.
\newblock Springer New York, 2020.

\bibitem{Last2017}
G{\"u}nter Last and Mathew Penrose.
\newblock {\em Lectures on the Poisson Process}.
\newblock Cambridge University Press, 2017.

\bibitem{Meester1996a}
Ronald Meester and Rahul Roy.
\newblock {\em Continuum Percolation}.
\newblock Cambridge University Press, 1996.

\bibitem{Roy2002}
Rahul Roy and Hideki Tanemura.
\newblock Critical intensities of boolean models with different underlying convex shapes.
\newblock {\em Advances in Applied Probability}, 34(1):48--57, 2002.

\bibitem{Sarkar1997}
Anish Sarkar.
\newblock Continuity and convergence of the percolation function in continuum percolation.
\newblock {\em Journal of Applied Probability}, 34(2):363--371, 1997.

\bibitem{Teixeira2017}
Augusto Teixeira and Daniel Ungaretti.
\newblock Ellipses percolation.
\newblock {\em Journal of Statistical Physics}, 168(2):369--393, 2017.

\end{thebibliography}

\end{document}